\documentclass[10pt,leqno]{article}

\usepackage{amsmath,amssymb,amsthm,mathrsfs,dsfont}
\usepackage{amssymb}
\usepackage{enumerate}
\makeatletter

\newcommand{\Rmnum}[1]{\expandafter\@slowromancap\romannumeral #1@}
\makeatother

\usepackage[margin=3cm]{geometry} 

\usepackage{titlesec,hyperref}

\usepackage{color}

\usepackage{fancyhdr}
\titleformat{\subsection}{\it}{\thesubsection.\enspace}{1pt}{}

\newtheorem{theo}{Theorem}[section]
\newtheorem{lemm}[theo]{Lemma}
\newtheorem{defi}[theo]{Definition}

\newtheorem{rema}[theo]{Remark}
\newtheorem*{nota}{Notation}
\numberwithin{equation}{section}

\allowdisplaybreaks 

\newcommand\w{{\widetilde{P}}}
\newcommand\R{{\mathbb{R}}}

\begin{document}
\title{Global conservative weak solutions and global strong solutions for a class of weakly dissipative nonlinear dispersive wave equations
	\hspace{-4mm}
}

\author{Yiyao $\mbox{Lian}^1$ \footnote{Email: lianyy7@mail2.sysu.edu.cn},\quad
	Zhenyu $\mbox{Wan}^1$ \footnote{Email: wanzhy9@mail2.sysu.edu.cn},\quad
	Zhaoyang $\mbox{Yin}^{1}$\footnote{E-mail: mcsyzy@mail.sysu.edu.cn}\\
$^1\mbox{School}$ of Science,\\ Shenzhen Campus of Sun Yat-sen University, Shenzhen 518107, China}

\date{}
\maketitle
\hrule

\begin{abstract}
	In this paper, we study the global existence of solutions of the Cauchy problem for a class of weakly dissipative nonlinear dispersive wave equations $u_t-u_{xxt}+(f\left(u\right))_x-(f\left(u\right))_{xxx}+\left(g\left(u\right)+\frac{f^{\prime\prime}\left(u\right)}{2}u_x^2\right)_x+\lambda\left(u-u_{xx}\right)=0$. This includes the weakly dissipative Camassa-Holm equation and the weakly dissipative hyperelastic rod wave equation as special cases. Specifically, we establish three global existence results: one concerning the energy conservative weak solutions in a time-weighted $H^1$ space, and the other two concerning strong solutions, which include the cases of small initial data and sign-changing initial data. Our results recover and extend many known results for several classical models.
	
\vspace*{5pt}
\noindent {\it 2020 Mathematics Subject Classification}: 35G25, 35A01, 35D30,
	
	\vspace*{5pt}
	\noindent{\it Keywords}: Weakly dissipative nonlinear dispersive wave equations; Weakly dissipative Camassa-Holm equation; Global conservative weak solutions; Global strong solutions.
\end{abstract}

\vspace*{10pt}

\tableofcontents

	\section{Introduction}\par
	
	In this paper, we consider the Cauchy problem for a class of nonlinear dispersive wave equations with dissipative term on the real line in the following form \cite{Novruzov2019}
	\begin{equation}\label{eq;u0}
		u_t-u_{xxt}+(f\left(u\right))_x-(f\left(u\right))_{xxx}+\left(g\left(u\right)+\frac{f^{\prime\prime}\left(u\right)}{2}u_x^2\right)_x+\lambda\left(u-u_{xx}\right)=0,
	\end{equation}
	with initial data
	\begin{equation}\label{initial data}
	u(0,x)=\bar{u}(x),
	\end{equation}
	where the dissipative coefficient $\lambda>0$, and the functions $f,g\in C^\infty(\R,\R)$ satisfy $f''\neq 0$ and $g(0)=0$. This equation covers the dissipative Camassa-Holm equation \cite{Camassa1993integrable}, the hyperelastic-rod wave equation \cite{Dai1998model}, and its generalization \cite{Coclite2005global,Holden2007global}, with the addition of the weakly dissipative term $\lambda\left(u-u_{xx}\right)$. Recently, Novruzov and Yazar \cite{Novruzov2019} established new local-in-space blow-up results that simplify and extend earlier blow-up criteria for equation \eqref{eq;u0}.
    
	If $\lambda=0,~f(u)=\frac{1}{2}u^2,~g(u)=2ku+u^2$, \eqref{eq;u0} becomes the dispersive Camassa-Holm equation
	\begin{align}\label{eq;CH}
	u_t-u_{xxt}+2ku_x+3uu_x=2u_xu_{xx}+uu_{xxx}
	\end{align}
	where $u$ represents the fluid velocity in the $x$ direction and $k$ is a dispersive coefficient associated with  the critical shallow water wave speed. This equation was originally derived by Camassa and Holm \cite{Camassa1993integrable}. It is obtained by performing an asymptotic expansion directly in the Hamiltonian for Euler's equations in the shallow water regime. For $k=0$, equation \eqref{eq;CH} reduces to the celebrated Camassa-Holm (CH) equation. As is well known, the CH equation admits a rich mathematical structure. Notably, it is completely integrable \cite{camassa1994new,Constantin2001scattering,Constantin1999circle}, admits a bi-Hamiltonian formulation \cite{Constantin1997hamiltonian,Fuchssteiner1981symplectic}, and hence possesses an infinite number of conservation laws in involution. It also admits peaked soliton solutions (peakons) of the form $ce^{-|x-ct|}$ \cite{Camassa1993integrable}, which describe traveling waves of maximal amplitude \cite{Constantin2006trajectories,Constantin2000stability} and are orbitally stable \cite{Constantin2000stability}. Moreover, the CH equation captures the wave-breaking phenomenon \cite{Beals2000multipeakons,Constantin1998wave}. Its local well-posedness in Sobolev spaces and Besov spaces has also been established \cite{Constantin1998well,Danchin2001,Danchin2003,Li.J2016}. The existence and uniqueness of global weak solutions were established in \cite{Con1,Xin2000}. Global conservative and dissipative solutions were subsequently studied in \cite{Constantin2007,Bressan2007}.
    
	If $\lambda=0,~f(u)=\frac{\gamma}{2}u^2,~g(u)=\frac{3-\gamma}{2}u^2$, \eqref{eq;u0} becomes the hyperelastic rod wave equation
	\begin{align}\label{eq;HR}
	u_t-u_{xxt}+3uu_x-\gamma(2u_xu_{xx}+uu_{xxx})=0,
	\end{align}
	which was derived by Dai \cite{Dai1998model}. This equation describes the behavior of far-field, finite-length, finite-amplitude radial deformation waves in cylindrical compressible hyperelastic rods. Here, $u$ represents the radial stretch relative to a pre-stressed state and $\gamma$ is related to the Finger deformation tensor. Stability of solitary waves for this equation was studied in \cite{constantin2000}. The local well-posedness, global solutions and blow-up phenomena of the Cauchy problem for \eqref{eq;HR} have been
	investigated by Yin \cite{Yin2004}.\par 
	If $\lambda=0$, equation \eqref{eq;u0} is called the generalized hyperelastic rod wave equation
	 	\begin{equation}\label{eq;ghrwe}
	 u_t-u_{xxt}+(f\left(u\right))_x-(f\left(u\right))_{xxx}+\left(g\left(u\right)+\frac{f^{\prime\prime}\left(u\right)}{2}u_x^2\right)_x=0,
	 \end{equation}
	 which was proposed in \cite{Holden2007global}. It was shown in \cite{Tian2014} that the Cauchy problem for \eqref{eq;ghrwe}  is locally well-posed in the Sobolev space $H^s$
	for $s > \frac 32$. Moreover, the blow-up phenomena were studied in \cite{LiMin2024,Novruzov2017, Tian2014}, and the blow-up rate was obtained in \cite{LiMin2024} for strong solutions to the Cauchy problem associated with \eqref{eq;ghrwe}. Furthermore, the existence of global conservative weak solutions to \eqref{eq;ghrwe} was established in \cite{Holden2007global,Zhou2024}.\par 
	In general, energy dissipation mechanisms are unavoidable in most real-world processes. In recent years, equations with dissipation have attracted increasing attention from researchers. For instamce, Wu and Yin \cite{Yin2009} considered the global existence and wave-breaking phenomena to the weakly dissipative Camassa-Holm equation
    \begin{align}\label{eq;weakly dissipative CH}
	u_t-u_{xxt}+3uu_x+\lambda(u-u_{xx})=2u_xu_{xx}+uu_{xxx}.
	\end{align}
    For more studies on weakly dissipative CH-type equations, we refer the reader to \cite{Freire2020,Freire2023, Ji2022, Lenells2013}.

    \subsection{Main results and strategy}
    
    It is known that solutions to equation \eqref{eq;u0} may develop singularities in finite time when the initial derivative is sufficiently negative at some point, as shown by Novruzov and Yazar in \cite{Novruzov2019}. A natural question is: under what conditions do global-in-time solutions exist? In this paper, we address this question by studying the Cauchy problem \eqref{eq;u0}–\eqref{initial data}, or its equivalent form: 
	\begin{equation}\label{eq;u1}
	\left\{\begin{array}{l}
	u_{t}+f'(u)u_x+\lambda u+P_x=0,\\
	u|_{t=0}=\bar{u},
	\end{array}\right.
	\end{equation}
	where $p(x)=\frac{1}{2}e^{-|x|}$ and $P\triangleq p*\Big(g(u)+\frac{f^{''}(u)}{2}u_x^2\Big).$
    We establish the existence of global conservative weak solutions, as well as strong solutions corresponding to small or sign-changing initial data. A distinguishing feature of \eqref{eq;u1} is that the linear dissipation induces an exponentially time-weighted $H^1$ conservation law. Specifically, while the $H^1$-norm of the solution decays over time, the corresponding time-weighted energy remains invariant, yielding $e^{\lambda t}\|u(t)\|_{H^1}=\|\bar{u}\|_{H^1}$ for smooth solutions. This property motivates us to to investigate the existence of global conservative weak solutions for this model in a time-weighted $H^1$ space. In addition to the weak solutions, we address the global strong solutions via two different mechanisms: the case of small initial data is handled by $H^s$-energy estimates, while the case of sign-changing initial data heavily exploits this conservation law to provide the essential \textit{a priori} bounds required to preclude finite-time blow-up.
    
    Now, we give the definition of a global conservative weak solution for the Cauchy problem \eqref{eq;u1} as follows.
	\begin{defi}\label{def2.1}
		Let $\bar{u}\in H^1(\mathbb{R})$. Then $u(t,x)$ is a global conservative weak solution to the Cauchy problem \eqref{eq;u1}, if $u(t,x)$ is a H$\ddot{o}$lder continuous function defined on $\R^+\times \mathbb{R}$ such that the following properties hold:
		\begin{enumerate}[(1)]
			\item At each fixed $t\geq 0$, $u(t,\cdot)\in H^1(\mathbb{R})$;
			\item The map $t\mapsto u(t,\cdot)$ is Lipschitz continuous from $\R^+$ into $L^2(\mathbb{R})$;
			\item $u(0,x)=\bar{u}(x)$ for $x\in\mathbb{R}$ and 
			\begin{equation}\label{eq;solution_def}
				\frac{\mathrm{d}}{\mathrm{d}t}(e^{\lambda t}u)=-f'(u)e^{\lambda t}u_x-e^{\lambda t}P_x
			\end{equation}
			holds for a.e. $t\geq 0$. Here \eqref{eq;solution_def} is understood as an equality  between functions in $L^2(\mathbb{R})$;
			\item $e^{2\lambda t}\int_{\mathbb{R}}(u^2(t,x)+u_x^2(t,x))\,\mathrm{d}x$ is conserved for a.e. $t\geq 0$.
		\end{enumerate}
	\end{defi}
    We state our main result on the global  conservative weak solutions to \eqref{eq;u1}.
    \begin{theo}\label{Thm 3.3}
	Let $\bar{u}\in H^1(\mathbb{R})$. Then the Cauchy problem \eqref{eq;u1} has a global conservative weak solution in the sense of Definition \ref{def2.1}. Furthermore, if there is a sequence of the initial value \(\bar{u}_{n}\) satisfying \(\| \bar{u}_{n}-\bar{u}\|_{H^{1}}\to 0\) as \(n\to \infty\), then the corresponding solution \(u_{n}(t,x)\) converges to \(u(t,x)\) uniformly for \((t,x)\) in any bounded sets.
    \end{theo}
    The construction of global conservative solutions to \eqref{eq;u1} is inspired by the work of Bressan and Constantin \cite{Constantin2007} on the Camassa-Holm equation. To this end, we introduce a new set of variables to transform the original equation into an equivalent semilinear system, establish the global existence of its solutions, and subsequently revert to the original coordinates to obtain the global conservative weak solutions. It is worth noting that while global conservative weak solutions only require the initial data to satisfy $\bar{u}\in H^1(\R)$, establishing the global existence of strong solutions demands a higher regularity of $\bar{u}\in H^s(\R)$ with $ s>\frac{3}{2},$ as well as additional conditions on $\bar{u}, \lambda, f$, and $g$.

    For strong solutions to \eqref{eq;u1}, we establish two main theorems concerning their global existence. The first theorem deals with the case of small initial data, requiring the dissipative coefficient $\lambda\ge M$ for some constant $M>1$. The proof is based on a bootstrap argument relying on $H^s$-energy estimates. For sufficiently small initial data, the strong linear dissipation dominates the nonlinear effects, which enables us to close \textit{a priori} estimates and extend the local solution globally in time. The second theorem addresses the case of sign-changing initial data, for which $f$ and $g$ must satisfy 
    $$f'''\equiv 0,\quad f''\geq\gamma>0\quad \text{and}\quad g'(u)=2f''(u)u.$$
    The key idea of the proof is that the sign-preservation property of the momentum $m$ along the characteristics, combined with the exponentially weighted energy conservation law, yields a uniform lower bound for $u_x$. According to the blow-up criterion, this uniform lower bound effectively rules out the wave-breaking phenomenon, thereby guaranteeing the global existence of the strong solution. The corresponding results on global strong solutions can be summarized as follows.
    \begin{theo}\label{Thm 4.1}
	Let $\bar{u}\in H^s(\mathbb{R})$ with $s>\frac{3}{2}$. There exist constants $0<\varepsilon<1$ and $M>1$ such that if
    \begin{align*}
        \|\bar{u}\|_{H^s}\le\varepsilon \quad \text{and} \quad\lambda\geq M.
    \end{align*}
    Then there exists a unique global strong solution $u\in C([0,\infty);H^s(\mathbb{R}))$  for the Cauchy problem \eqref{eq;u1}.
\end{theo}
    \begin{rema}
        A result analogous to that in Theorem \ref{Thm 4.1} holds for the weakly dissipative Camassa–Holm equation \eqref{eq;weakly dissipative CH}, where the assumption on the dissipative coefficient is less restrictive; in particular, the condition $\lambda \ge M$ is no longer needed.
    \end{rema}
    \begin{theo}\label{Thm 4.6}
     Let $\bar{u}\in H^s(\R)$ with $s>\frac 32$. Assume that $f, g \in C^\infty (\R)$ satisfy condition 
     \begin{equation}\label{condi-fg}
          f'''\equiv 0,\quad f''\geq \gamma>0\quad\text{and}\quad g'(u)=2f''(u)u.
     \end{equation}
    Suppose that there exists $x_0\in\R$ such that
     \begin{equation}
         \bar{m}(x)\leq 0,\quad \text{if}~x\in (-\infty, \bar{y}(x_0)]\quad \text{and}\quad  \bar{m}(x)\geq 0,\quad \text{if}~x\in [\bar{y}(x_0),+\infty).
     \end{equation}
     Then the corresponding solution $u(t, x)$ for the Cauchy problem \eqref{eq;u1} exists globally in time. Moreover, the global solution decays to 0 in the $H^s$ norm as time goes to infinity.
    \end{theo}
    \begin{rema}
       There exist many equations satisfying condition \eqref{condi-fg}, such as the classical weakly dissipative Camassa-Holm equation \eqref{eq;weakly dissipative CH} with $f(u)=\frac{1}{2}u^2$ and $g(u)=u^2$.  Therefore, Theorem \ref{Thm 4.6} actually generalizes the result of Theorem 3.1 in \cite{Yin2009}.
    \end{rema}

    \subsection{Organization and notation}
    
    The paper is organized as follows. In Section \ref{sec-pre}, we give some preliminaries which will be used in the sequel. In Section \ref{sec-weak}, we establish the existence of  global conservative weak solutions to the Cauchy problem \eqref{eq;u1}. In Section \ref{sec-strong}, we obtain two results concerning the global existence of strong solutions to the Cauchy problem \eqref{eq;u1}.

    \begin{nota}
        The notation $a\lesssim b$ means that $a\leq Cb$ for some uniform constant $C$, which may differ from line to line. Let $[\cdot,\cdot]$ denote the commutator of $A$ and $B$, i.e., $[A,B]=AB-BA$. We also use $\langle,\rangle$ to denote the standard inner product in $L^2(\mathbb{R})$.
    \end{nota}
	
	\section{\textbf{Preliminaries}}\label{sec-pre}
    In this section, we derive some time-weighted energy conservation laws and some useful lemmas. 
    
    Firstly, setting $k=e^{\lambda t}u$, the Cauchy problem \eqref{eq;u1} can be rewritten equivalently as follows:
	\begin{equation}\label{eq;k}
	\left\{\begin{array}{l}
	k_t+f'(e^{-\lambda t}k)k_x+\widetilde{P}_x=0,\\
	k|_{t=0}=\bar{u},
	\end{array}\right.
	\end{equation}
	where $\widetilde{P}\triangleq p*\Big(e^{\lambda t}g(e^{-\lambda t}k)+e^{-\lambda t}\frac{f^{''}(e^{-\lambda t}k)}{2}k_x^2\Big).$

	For smooth solutions, differentiating \eqref{eq;k} with respect to $x$, we have
	\begin{equation}\label{eq;k_xt}
		k_{xt}+f'(e^{-\lambda t}k)k_{xx}=e^{\lambda t}g(e^{-\lambda t}k)-\frac{1}{2}e^{-\lambda t}f''(e^{-\lambda t}k)k_x^2-\widetilde{P}.
	\end{equation}
	Multiplying \eqref{eq;k} by $k$ and \eqref{eq;k_xt} by $k_x$, we obtain 
	\begin{equation}\label{eq;(k^2)_t}
		\frac{1}{2}(k^2)_t+\Big(\frac12 f'(e^{-\lambda t}k)k^2\Big)_x+(k\widetilde{P})_x=\frac12e^{-\lambda t}f''(e^{-\lambda t}k)k^2k_x+k_x\widetilde{P}
	\end{equation}and
	\begin{equation}\label{eq;(k_x^2)_t}
	\frac{1}{2}(k^2_x)_t+\Big(\frac12 f'(e^{-\lambda t}k)k_x^2\Big)_x=e^{-\lambda t}g(e^{-\lambda t}k)k_x-k_x\widetilde{P}.
	\end{equation}
	Define 
	$$H(k)=\int^k_0\Big(\frac12e^{-\lambda t}f''(e^{-\lambda t}s)s^2+e^{-\lambda t}g(e^{-\lambda t}s)\Big)\,\mathrm{d}s.$$
	Together with \eqref{eq;(k^2)_t} and \eqref{eq;(k_x^2)_t}, this gives that the total energy
	\begin{equation*}
		E(t)=e^{2\lambda t}\int_{\mathbb{R}}(u^2(t,x)+u_x^2(t,x))dx=\int_{\mathbb{R}}(k^2(t,x)+k_x^2(t,x))\,\mathrm{d}x
	\end{equation*}
	is constant in time.
	Since $f,g\in C^\infty$ and $g(0)=0$, we have
	\begin{equation}\label{eq;g_estimate}
		e^{\lambda t}|g(e^{-\lambda t}k)|\leq\sup_{|s|\leq\| k\|_{L^\infty}}|g^{\prime}(s)||k(x)|\leq C(\|k\|_{H^1})|k(x)|,
	\end{equation}
	\begin{equation}\label{eq;f_estimate}
		|f(e^{-\lambda t}k)|\leq C(\|k\|_{H^1}).
	\end{equation}
	Since $\widetilde{P}, \widetilde{P}_x$ are convolutions of $e^{\lambda t}g(e^{-\lambda t}k)+e^{-\lambda t}\frac{f^{''}(e^{-\lambda t}k)}{2}k_x^2$, from the above bound
	on the total energy, we easily derive that
	$$\|\widetilde{P}(t)\|_{L^\infty}, \|\widetilde{P}_x(t)\|_{L^\infty}, \|\widetilde{P}(t)\|_{L^2}, \|P_x(t)\|_{L^2}\lesssim E(0).$$
	
    Next, we recall the known local well-posedness results and necessary and sufficient condition for the blow-up of
the solutions for equation \eqref{eq;u1}.
    \begin{lemm}\cite{Novruzov2019,Tian2014}\label{le4.2}
    Assume  $f, g \in C^\infty (\R)$. Let $\bar{u}\in  H^s (\R)$  with $s > \frac 32$. Then there exists $T > 0$, 
    with $T = T (\bar{u}, f, g)$ and a unique solution $u$ to the Cauchy problem \eqref{eq;u1} such that $u\in C([0, T); H^s(\R)) \cap C^1 ([0, T), H^{s-1}(\R))$. The solution has energy integral which satisfies
    \begin{equation*}
        \int_{\R} (u^2+u_x^2)dx=e^{-2\lambda t}\int_{\R} (\bar{u}^2+\bar{u}_x^2) \,\mathrm{d}x.
    \end{equation*}
\end{lemm}
\begin{lemm}\cite{Novruzov2019,Tian2014}\label{le4.3}
    Assume  $f, g \in C^\infty (\R)$ and $f''\geq \gamma>0.$ Let $\bar{u}\in  H^s (\R)$ with $s > \frac 32$. Then the solution of \eqref{eq;u1} blows up in a finite time $T^* > 0$ if and only if
    \begin{equation}\label{con_blow up}
        \liminf_{t\rightarrow T^*}\{ \inf_{x\in\R}u_x(t,x)\}=-\infty.
    \end{equation}
\end{lemm}
As a preparation, we present the following lemmas.
    \begin{lemm}\cite{ConstantinBoussinesq}\label{lem-Fu-Fv}
       Let \(F\in C^{\infty}(\mathbb{R})\). If \(u,v\in H^{s}(\mathbb{R})\cap L^{\infty}(\mathbb{R})\), \(s\geq 0\), then
\[
\| F(u)-F(v)\|_{H^{s}(\mathbb{R})\cap L^{\infty}(\mathbb{R})}\leq K\| u-v\|_{H^{s}(\mathbb{R})\cap L^{\infty}(\mathbb{R})},
\]
where \(K\) depends on $\| u\|_{H^{s}}$ and $\| v\|_{H^{s}}$. In particular, when \(s>\frac{1}{2}\), we have
\[
\| F(u)-F(v)\|_{H^{s}(\mathbb{R})}\leq 2K\| u-v\|_{H^{s}(\mathbb{R})}. \]
    \end{lemm}

    \begin{lemm}\cite{ConstantinBoussinesq}\label{lem-F0}
      Let \(F\in C^{\infty}(\mathbb{R})\). If \(u\in H^{s}(\mathbb{R})\cap L^{\infty}(\mathbb{R})\), \(s\geq 0\), and \(F(0)=0\), then
\[
\| F(u)\|_{H^{s}(\mathbb{R})\cap L^{\infty}(\mathbb{R})}\leq K\| u\|_{H^{s}(\mathbb{R})\cap L^{\infty}(\mathbb{R})},
\]
where \(K\) depends on \(\| u\|_{H^{s}}\). In particular, when \(s>\frac{1}{2}\), we have
\[
\| F(u)\|_{H^{s}(\mathbb{R})}\leq 2K\| u\|_{H^{s}(\mathbb{R})}.\]   
    \end{lemm}

\begin{lemm}\cite{Kato}\label{lem-commu}
	Let $s>0$ and $\Lambda^s=(1-\partial^2_{x})^{\frac{s}{2}}$, then
	\begin{align*}
		\|[\Lambda^s,u]v\|_{L^2(\mathbb{R})}\le C\left(\|u_x\|_{L^{\infty}(\mathbb{R})}\|\Lambda^{s-1}v\|_{L^2(\mathbb{R})}+\|\Lambda^su\|_{L^{2}(\mathbb{R})}\|v\|_{L^{\infty}(\mathbb{R})}\right),
	\end{align*}
	where $C$ is a constant independent of $u$ and $v$.
\end{lemm}

\section{Global conservative weak solutions}\label{sec-weak}

\subsection{An equivalent semilinear system}
In this section, we will derive a semilinear system by introducing some new variables. Let $\bar{u}\in H^1(\mathbb{R})$ be the initial data. We consider the variable $\xi \in \mathbb{R}$ and define the non-decreasing map $\xi \mapsto \bar{y}(\xi)$ via
\begin{equation}\label{bary_def}
	\int_0^{\bar{y}(\xi)}(1+\bar{u}_x^2)\,\mathrm{d}x=\xi,
\end{equation}
which implies that
\begin{equation}\label{eq;bar y xi}
\bar{y}_\xi(\xi)(1+\bar{u}_x^2(\bar{y}(\xi)))=1.
\end{equation}
Assume the solution $u$ to equation \eqref{eq;u1} is Lipschitz continuous on $[0, T]$. Using the independent variables $(t, \xi)$, we now derive an equivalent system. Let $t\mapsto y(t, \xi)$ be the characteristic starting at $\bar{y}(\xi)$ satisfying
\begin{equation}\label{eq;y_t}
\frac{\partial}{\partial t}y(t,\xi)=f'(e^{-\lambda t}k(t,y(t,\xi))),\quad y(0,\xi)=\bar{y}(\xi).
\end{equation}
Hereafter, we denote
$$ k(t,\xi)=k(t,y(t,\xi)),\quad \widetilde{P}(t,\xi)=\widetilde{P}(t,y(t,\xi)),\quad \widetilde{P}_x(t,\xi)=\widetilde{P}_x(t,y(t,\xi)), $$
and define
\begin{equation}\label{eq;v_def,q_def}
	v(t,\xi)=2\arctan k_x(t,\xi),\quad q(t,\xi)=(1+k_x^2(t,\xi))\frac{\partial y}{\partial \xi},
\end{equation}
where $k_x(t,\xi)=k_x(t,y(t,\xi))$ and $q(0,\xi)=1$. The following identities will be useful in the sequel:
\begin{equation}\label{eq;identities}
	\frac{1}{1+k_{x}^{2}}=\cos^{2}\frac{v}{2},\quad
	\frac{k_{x}^{2}}{1+k_{x}^{2}}=\sin^{2}\frac{v}{2},\quad
	\frac{k_{x}}{1+k_{x}^{2}}=\frac{1}{2}\sin v,\quad 
	\frac{\partial y}{\partial\xi}=\frac{q}{1+k_{x}^{2}}=\cos^{2}\frac{v}{2}\cdot q,
\end{equation}
  which immediately implies that 
  \begin{equation}\label{eq;y(t,xi')-y(t,xi)}
  y(t,\xi^{\prime})-y(t,\xi)=\int_\xi^{\xi^{\prime}}\cos^2\frac{v(t,s)}{2}\cdot q(t,s)\,\mathrm{d}s.
  \end{equation}
  Thanks to \eqref{eq;identities} and \eqref{eq;y(t,xi')-y(t,xi)}, together with the change of variables $x=y(t,\xi')$, we can express 
$\widetilde{P}$ and $\widetilde{P}_x$ in terms of the new variable $\xi$, namely,
\begin{align}\label{p(t,xi)}
\w(t,\xi)&=\frac{1}{2}\int^{+\infty}_{-\infty}e^{-|y(t,\xi)-x|}\Big(e^{\lambda t}g(e^{-\lambda t}k)+\frac12 e^{-\lambda t}f''(e^{-\lambda t}k)k_x^2\Big)\,\mathrm{d}x\notag\\
&=\frac12 \int^{+\infty}_{-\infty}e^{-|\int^{\xi'}_{\xi}\cos^2\frac{v(s)}{2}q(s)ds|}\Big(e^{\lambda t}g(e^{-\lambda t}k)\cos^2\frac{v}{2}+\frac12 e^{-\lambda t}f''(e^{-\lambda t}k)\sin^2\frac{v}{2}\Big)q\,\mathrm{d}\xi'
\end{align}
and
\begin{align}\label{p_x(t,xi)}
\w_x(t,\xi)&=\frac{1}{2}\Big(\int^{+\infty}_{y(t,\xi)}-\int^{y(t,\xi)}_{-\infty}\Big) e^{-|y(t,\xi)-x|}\Big(e^{\lambda t}g(e^{-\lambda t}k)+\frac12 e^{-\lambda t}f''(e^{-\lambda t}k)k_x^2\Big)\,\mathrm{d}x\notag\\
&=\frac12 \Big(\int^{+\infty}_{\xi}-\int^{\xi}_{-\infty}\Big) e^{-|\int^{\xi'}_{\xi}\cos^2\frac{v(s)}{2}q(s)ds|}\Big(e^{\lambda t}g(e^{-\lambda t}k)\cos^2\frac{v}{2}+\frac12 e^{-\lambda t}f''(e^{-\lambda t}k)\sin^2\frac{v}{2}\Big)q\,\mathrm{d}\xi'.
\end{align}

We now derive the evolution equations for $(k, v, q)$ in terms of the variables $(t,\xi)$, which form a semilinear system. Firstly, for $k$, we have
\begin{equation}\label{eq;k_t}
\frac{\partial}{\partial t}k(t,\xi)=k_t+f'(e^{-\lambda t}k)k_x=-\w_x(t,\xi).
\end{equation}
Secondly, we derive an evolution equation for $q$ in \eqref{eq;v_def,q_def}. Using \eqref{eq;(k_x^2)_t}, \eqref{eq;y_t} and \eqref{eq;identities}, we obtain
\begin{align}\label{eq;q_t}
	\frac{\partial}{\partial t}q(t,\xi)&=\frac{\partial}{\partial t}(1+k_x^2(t,y(t,\xi)))\frac{\partial y}{\partial t}+(1+k_x^2)\frac{\partial^2 y}{\partial t\partial\xi}\notag\\
	&=\Big[(1+k_x^2)_t+(1+k_x^2)_xf'(e^{-\lambda t}k)+(1+k_x^2)f''(e^{-\lambda t}k)e^{-\lambda t}k_x \Big]\cdot\frac{\partial y}{\partial \xi} \notag\\
	&=\Big[(1+k_x^2)_t+((1+k_x^2)f'(e^{-\lambda t}k))_x \Big]\cdot\frac{q}{1+k_x^2} \notag\\
	&=\Big[(f'(e^{-\lambda t}k))_x+2e^{\lambda t}g(e^{-\lambda t}k)k_x-2k_x\w \Big]\cdot\frac{q}{1+k_x^2} \notag\\
	&=\Big[(f''(e^{-\lambda t}k)e^{-\lambda t}+2e^{\lambda t}g(e^{-\lambda t}k)-2\w \Big]k_x\cdot\frac{q}{1+k_x^2}\notag\\
	&=\frac12 \sin v\cdot q\Big[(f''(e^{-\lambda t}k)e^{-\lambda t}+2e^{\lambda t}g(e^{-\lambda t}k)-2\w \Big].
\end{align}
Finally, using \eqref{eq;k_xt} and \eqref{eq;identities} yields
\begin{align}\label{eq;v_t}
	\frac{\partial}{\partial t}v(t,\xi)&=\frac{2}{1+k_x^2}(k_{xt}+k_{xx}f'(e^{-\lambda t}k))\notag\\
	&=\frac{2}{1+k_x^2}(e^{\lambda t}g(e^{-\lambda t}k)-\frac{1}{2}e^{-\lambda t}f''(e^{-\lambda t}k)k_x^2-\widetilde{P})\notag\\
	&=2(e^{\lambda t}g(e^{-\lambda t}k)-\w)\cos^2\frac{v}{2}-e^{-\lambda t}f''(e^{-\lambda t}k)\sin^2\frac{v}{2}.
\end{align}
Therefore, combining \eqref{p(t,xi)}-\eqref{eq;v_t}, a semilinear system for $(k,v,q)$ in terms of $(t,\xi)$ is derived.

\subsection{Global solutions of the semilinear system}
In this section, we prove the global existence of solutions for the semilinear system. Given the initial data $\bar{u}(x)\in H^1(\mathbb{R})$, we can transfer the Cauchy problem \eqref{eq;k} into the following semilinear system
\begin{equation}\label{eq;semilinear system}
\left\{\begin{array}{l}
\frac{\partial k}{\partial t}=-\w_x,\\
\frac{\partial v}{\partial t}=2(e^{\lambda t}g(e^{-\lambda t}k)-\w)\cos^2\frac{v}{2}-e^{-\lambda t}f''(e^{-\lambda t}k)\sin^2\frac{v}{2},\\
\frac{\partial q}{\partial t}=\frac12 \sin v\cdot q\Big((f''(e^{-\lambda t}k)e^{-\lambda t}+2e^{\lambda t}g(e^{-\lambda t}k)-2\w \Big),
\end{array}\right.
\end{equation}
with initial data
\begin{equation}\label{eq;initial data of semilinear system}
\left\{\begin{array}{l}
k(0,\xi)=\bar{u}(\bar{y}(\xi)),\\
v(0,\xi)=2\arctan\bar{u}_x(\bar{y}(\xi)),\\
q(0,\xi)=1,\\
\end{array}\right.
\end{equation}
where $\w$ and $\w_x$  are defined according to \eqref{p(t,xi)} and \eqref{p_x(t,xi)}. We regard
\eqref{eq;semilinear system} as an ordinary differential equation in the Banach space
$$X\triangleq H^1(\mathbb{R})\times\left(L^2(\mathbb{R})\cap L^\infty(\mathbb{R})\right)\times L^\infty(\mathbb{R})$$
with norm
$$\left\|(k,v,q)\right\|_X\triangleq\|u\|_{H^1}+\|v\|_{L^2}+\|v\|_{L^\infty}+\|q\|_{L^\infty}.$$
Our main theorem can be stated as follows.

\begin{theo}\label{Thm 3.1}
	Let $\bar{u}\in H^1(\mathbb{R})$. Then the Cauchy problem \eqref{eq;semilinear system}-\eqref{eq;initial data of semilinear system} has a unique
	solution, defined for all times $t\geq 0$.
\end{theo}
\begin{proof}
	Now we prove Theorem \ref{Thm 3.1} in two steps.  Our goal of the first step is to prove the local existence and uniqueness of solutions to the Cauchy problem \eqref{eq;semilinear system}-\eqref{eq;initial data of semilinear system}. By the standard theory of ordinary differential equations in Banach spaces, it suffices to show that the right-hand side of \eqref{eq;semilinear system} is locally Lipschitz continuous. In the second step, using the conservation of energy, we can then extend the local solution globally in time.
	
	\textbf{Step 1: Local existence}\par
	Define 
	$$\Omega=\Big\{(k,v,q)\left| \|k\|_{H^1}\leq\alpha,~\|v\|_{L^2}\leq\beta,\|v\|_{L^\infty}\leq\frac{3\pi}{2}, q(x)\in[q^{-},q^{+}],~ a.e.~x\in \mathbb{R}\right. \Big\}, $$
	for any constants $\alpha,~\beta,~q^-,~q^+>0.$
	To establish the local existence of solutions, it suffices to
	show that mapping $(k, v, q)$ to the right hand side of $\eqref{eq;semilinear system}$ is Lipschitz continuous on every bounded domain $\Omega \subset X$.\par
	A direct calculation using Sobolev's inequality 
	\begin{equation}\label{eq;sobolev's inequality}
	\|k\|_{L^\infty(\mathbb{R})}\leq \|k\|_{H^1(\mathbb{R})},
	\end{equation}
	combined with \eqref{eq;g_estimate} and \eqref{eq;f_estimate}, shows that the maps 
	$$2e^{\lambda t}g(e^{-\lambda t}k)\cos^2\frac{v}{2}-e^{-\lambda t}f''(e^{-\lambda t}k)\sin^2\frac{v}{2}\quad\text{and}\quad \frac12 \sin v\cdot q\Big((f''(e^{-\lambda t}k)e^{-\lambda t}+2e^{\lambda t}g(e^{-\lambda t}k)\Big) $$
	are all Lipschitz continuous from $\Omega$ into $L^2(\mathbb{R})\cap L^\infty(\mathbb{R})$. Hence, we only need to prove the maps
	\begin{equation}\label{Lip Continuous}
	(k,v,q)\mapsto \w,\quad(k,v,q)\mapsto \w_x
	\end{equation}
	are Lipschitz continuous from $\Omega$ into $H^1(\mathbb{R})$, which implies the maps above are also Lipschitz continuous from $\Omega$ into $L^2(\mathbb{R})\cap L^\infty(\mathbb{R})$.
	Following \cite{Constantin2007}, we first derive some necessary estimates. For $(k,v,q)\in\Omega$, we have
	\begin{align*}
	\text{measure}\left\{\xi\in\mathbb{R}\Big| |\frac{v(\xi)}2|\geq\frac\pi4\right\} &\leq\text{measure}\left\{\xi\in \mathbb{R}\Big|\sin^{2}\frac{v(\xi)}2\geq\frac1{18}\right\}\notag\\
	&\leq 18\int_{\left\{\xi\in\mathbb{R}|\sin^2\frac{v(\xi)}{2}\geq\frac{1}{18}\right\}}\sin^2\frac{v(\xi)}{2}\,\mathrm{d}\xi\notag\\
	&\leq\frac{9}{2} \beta^2.
	\end{align*}
	Then, for any $\xi_1<\xi_2$, we have
	\begin{align}\label{eq;4.5}
	\int_{\xi_{1}}^{\xi_{2}}\cos^{2}\frac{v(\xi)}{2}\cdot q(\xi)\,\mathrm{d}\xi
	&\geq\int_{\left\{\xi\in[\xi_{1},\xi_{2}]||\frac{v(\xi)}{2}|\leq\frac{\pi}{4}\right\}}\cos^2\frac{v(\xi)}{2}q(\xi)\,\mathrm{d}\xi\notag\\
	&\geq\int_{\left\{\xi\in[\xi_{1},\xi_{2}]||\frac{v(\xi)}{2}|\leq\frac{\pi}{4}\right\}}\frac{q^{-}}{2}\,\mathrm{d}\xi\notag\\
	&\geq\left(\frac{\xi_{2}-\xi_{1}}{2}-\frac{9}{4}\beta^{2}\right)q^{-}.
	\end{align}
	Define the exponentially decaying function by
	$$\Gamma(\zeta)\triangleq\min\left\{1,\exp\left((\frac{9}{4}\beta^2-\frac{|\zeta|}{2})q^-\right)\right\}. $$
	A straightforward calculation yields:
	\begin{equation}\label{eq;Gamma_L^1}
	\|\Gamma\|_{L^1}=\left(\int_{|\zeta|\leq \frac92\beta^2}+\int_{|\zeta|\geq\frac92\beta^2}\right)\Gamma(\zeta)\,\mathrm{d}\zeta=9\beta^2+\frac{4}{q^-}<\infty
	\end{equation}and
	\begin{equation}\label{eq;Gamma||_L^1}
	\|\Gamma|\cdot|\|_{L^1}=\left(\int_{|\zeta|\leq \frac92\beta^2}+\int_{|\zeta|\geq\frac92\beta^2}\right)\Gamma(\zeta)|\zeta|\,\mathrm{d}\zeta=\frac{81}{2}\beta^4+\frac{18}{q^-}\beta^2+\frac{8}{(q^-)^2}<\infty.
	\end{equation}
	
In what follows, we prove that $\w,\w_x\in H^1(\mathbb{R})$, namely $\w,\partial_\xi \w,\w_x,\partial_\xi \w_x\in L^2(\R)$. From the definition \eqref{p_x(t,xi)} of $\w_x$, applying Young's inequality together with \eqref{eq;g_estimate}, \eqref{eq;f_estimate} and \eqref{eq;4.5}, it follows that
\begin{align}\label{est of P}
\|\w_{x}\|_{L^{2}}&\leq\|\frac{q^{+}}{2}\Big( \Gamma*|e^{\lambda t}g(e^{-\lambda t}k)\cos^2\frac{v}{2}+\frac12 e^{-\lambda t}f''(e^{-\lambda t}k)\sin^2\frac{v}{2} |\Big)\|_{L^{2}} \notag\\
&\leq\frac{q^{+}}{2}\|\Gamma\|_{L^{1}}\|e^{\lambda t}g(e^{-\lambda t}k)\cos^2\frac{v}{2}+\frac12 e^{-\lambda t}f''(e^{-\lambda t}k)\sin^2\frac{v}{2}\|_{L^{2}}\notag\\
&\leq Cq^{+} \|\Gamma\|_{L^{1}}\left(\|k\|_{L^{2}}+e^{-\lambda t}\|v\|_{L^{2}}\right)<\infty.
\end{align}
For $\partial_\xi \w_x$, differentiating \eqref{p_x(t,xi)} with respect to $\xi$ gives
\begin{align*}
	\partial\xi\w_{x}&=-\left( e^{\lambda t}g(e^{-\lambda t}k)\cos^2\frac{v}{2}+\frac12 e^{-\lambda t}f''(e^{-\lambda t}k)\sin^2\frac{v}{2}\right)(\xi)\cdot q(\xi)\notag\\
	&~~~+\frac{1}{2}\left(\int_{\xi}^{\infty}-\int_{-\infty}^{\xi}\right)e^{-\left|\int_{\xi}^{\xi^{\prime}}\cos^{2}\frac{v(s)}{2}q(s)ds\right|} \cos^2\frac{v(\xi)}{2}q(\xi)\mathrm{sign}(\xi^{\prime}-\xi)\notag\\
	&~~~\cdot\left( e^{\lambda t}g(e^{-\lambda t}k)\cos^2\frac{v}{2}+\frac12 e^{-\lambda t}f''(e^{-\lambda t}k)\sin^2\frac{v}{2}\right)(\xi')\cdot q\left(\xi^{\prime}\right)\,\mathrm{d}\xi^{\prime},
	\end{align*}
which implies 
\begin{align*}
|\partial\xi\w_{x}|&\leq q^+ |e^{\lambda t}g(e^{-\lambda t}k)\cos^2\frac{v}{2}+\frac12 e^{-\lambda t}f''(e^{-\lambda t}k)\sin^2\frac{v}{2}|\notag\\
&~~~+\frac{1}{2}(q^+)^2\Big(\Gamma *|e^{\lambda t}g(e^{-\lambda t}k)\cos^2\frac{v}{2}+\frac12 e^{-\lambda t}f''(e^{-\lambda t}k)\sin^2\frac{v}{2}|\Big).
\end{align*}
Young's inequality thus yields
\begin{equation}\label{est of partial_xiP}
\|\partial_\xi \w_{x}\|_{L^{2}}\leq Cq^{+}(1+ q^+\|\Gamma\|_{L^{1}})\left(\|k\|_{L^{2}}+e^{-\lambda t}\|v\|_{L^{2}}\right)<\infty.
\end{equation}
The estimates for $\w, \partial_\xi\w$ are similar to \eqref{est of P} and \eqref{est of partial_xiP}, so we omit the details here.\par

Next, to establish the Lipschitz continuity of \eqref{Lip Continuous}, it suffices to show that their partial derivatives
\begin{equation}\label{eq;operators}
\frac{\partial \w}{\partial k},\quad\frac{\partial \w}{\partial v},\quad\frac{\partial \w}{\partial q},\quad\frac{\partial \w_x}{\partial k},\quad\frac{\partial \w_x}{\partial v},\quad\frac{\partial \w_x}{\partial q},
\end{equation}
are uniformly bounded for $(k,v,q)\in\Omega$. For $\frac{\partial \w_x}{\partial k}$, given $(k,v,q)\in\Omega$, we define linear operators $\frac{\partial \w_x}{\partial k}, \frac{\partial (\partial_\xi \w_x)}{\partial k}$  by
\begin{align*}
&\left(\frac{\partial \w_{x}(k,v,q)}{\partial k}\cdot\phi\right)(\xi)\notag\\ &=\frac12\left(\int_{\xi}^{\infty}-\int_{-\infty}^{\xi}\right)e^{-\left|\int_{\xi}^{\xi^{\prime}}\cos^{2}\frac{v(s)}{2}\cdot q(s)\,\mathrm{d}s\right|} \Big(g'(e^{-\lambda t}k)\cos^2\frac{v}{2}+\frac12 e^{-2\lambda t}f'''(e^{-\lambda t}k)\sin^2\frac{v}{2}\Big)(\xi') q(\xi^{\prime})\phi(\xi^{\prime})\,\mathrm{d}\xi^{\prime}
\end{align*}
and
\begin{align*}
\left(\frac{\partial (\partial_\xi\w_{x})(k,v,q)}{\partial k}\cdot\phi\right)(\xi)&=-\left(g'(e^{-\lambda t}k)\cos^2\frac{v}{2}+\frac12 e^{-2\lambda t}f'''(e^{-\lambda t}k)\sin^2\frac{v}{2}\right)(\xi)\cdot q(\xi)\phi(\xi)\notag\\
&~~~+\frac{1}{2}\left(\int_{\xi}^{\infty}-\int_{-\infty}^{\xi}\right)e^{-\left|\int_{\xi}^{\xi^{\prime}}\cos^{2}\frac{v(s)}{2}q(s)\,\mathrm{d}s\right|} \cos^2\frac{v(\xi)}{2}q(\xi)\mathrm{sign}(\xi^{\prime}-\xi)\notag\\
&~~~\cdot\left( g'(e^{-\lambda t}k)\cos^2\frac{v}{2}+\frac12 e^{-2\lambda t}f'''(e^{-\lambda t}k)\sin^2\frac{v}{2}\right)(\xi')\cdot q\left(\xi^{\prime}\right)\phi(\xi^{\prime})\,\mathrm{d}\xi^{\prime},
\end{align*}
with $\phi\in H^1(\R)$. By virtue of estimates $|g'(e^{-\lambda t}k)|\leq C(\|k\|_{H^1}) $ and $|f'''(e^{-\lambda t}k)|\leq C(\|k\|_{H^1})$ together with Young's inequality, we deduce that
\begin{align}\label{eq;(4.9)}
\|\frac{\partial \w_{x}}{\partial k}\cdot\phi\|_{L^2}&\leq Cq^+\|\Gamma\|_{L^1}\Big(\|g'(e^{-\lambda t}k)\|_{L^\infty}+\|f'''(e^{-\lambda t}k)\|_{L^\infty}\Big)\|\phi\|_{L^2}\notag\\
&\leq Cq^+\|\Gamma\|_{L^1}\|\phi\|_{H^1}
\end{align}
and
\begin{align}\label{eq;(4.10)}
\|\frac{\partial (\partial_\xi\w_{x})}{\partial k}\cdot\phi\|_{L^2}&\leq Cq^+(1+q^+\|\Gamma\|_{L^1})\|\phi\|_{H^1}.
\end{align}
From \eqref{eq;Gamma_L^1}, \eqref{eq;(4.9)} and \eqref{eq;(4.10)}, it follows that
$\frac{\partial \w_x}{\partial k}$ is a bounded linear operator from $H^1(\R)$ to $H^1(\R)$.
 	For $\frac{\partial \w_x}{\partial q}$, given $(k,v,q)\in\Omega$, we define linear operators $\frac{\partial \w_x}{\partial q}, \frac{\partial (\partial_\xi \w_x)}{\partial q}$  by
	\begin{align*}
	\left(\frac{\partial \w_{x}(k,v,q)}{\partial q}\cdot\phi\right)(\xi)&=\frac12\left(\int_{\xi}^{\infty}-\int_{-\infty}^{\xi}\right)e^{-\left|\int_{\xi}^{\xi^{\prime}}\cos^{2}\frac{v(s)}{2}\cdot q(s)\,\mathrm{d}s\right|} \Big(-\Big|\int_{\xi}^{\xi^{\prime}}\cos^{2}\frac{v(s)}{2} ds\Big|\cdot q(\xi')+1\Big) \notag\\
	&~~~~\cdot\Big(e^{\lambda t}g(e^{-\lambda t}k)\cos^2\frac{v}{2}+\frac12 e^{-\lambda t}f''(e^{-\lambda t}k)\sin^2\frac{v}{2}\Big)(\xi')\phi(\xi^{\prime})\,\mathrm{d}\xi^{\prime}
	\end{align*}
    and
	\begin{align*}
	&\left(\frac{\partial (\partial_\xi\w_{x})(k,v,q)}{\partial q}\cdot\phi\right)(\xi)\notag\\
	&=-\left(e^{\lambda t}g(e^{-\lambda t}k)\cos^2\frac{v}{2}+\frac12 e^{-\lambda t}f''(e^{-\lambda t}k)\sin^2\frac{v}{2}\right)(\xi)\cdot \phi(\xi)\notag\\
	&~~~+\frac{1}{2}\left(\int_{\xi}^{\infty}-\int_{-\infty}^{\xi}\right)e^{-\left|\int_{\xi}^{\xi^{\prime}}\cos^{2}\frac{v(s)}{2}q(s)\,\mathrm{d}s\right|} \Big(-\Big|\int_{\xi}^{\xi^{\prime}}\cos^{2}\frac{v(s)}{2} ds\Big|\cdot q(\xi)q(\xi')+q(\xi)+q(\xi')\Big)\notag\\
	&~~~\cdot\cos^2\frac{v(\xi)}{2}\mathrm{sign}(\xi^{\prime}-\xi)\left( e^{\lambda t}g(e^{-\lambda t}k)\cos^2\frac{v}{2}+\frac12 e^{-\lambda t}f''(e^{-\lambda t}k)\sin^2\frac{v}{2}\right)(\xi')\cdot \phi(\xi^{\prime})\,\mathrm{d}\xi^{\prime},
	\end{align*}
	with $\phi\in H^1(\R)$. Applying Young's inequality, using \eqref{eq;g_estimate} and \eqref{eq;f_estimate} together with the embedding $H^1(\R)\hookrightarrow L^\infty(\R)$,  we deduce that
	\begin{align}\label{eq;(4.14)}
	\|\frac{\partial \w_{x}}{\partial q}\cdot\phi\|_{L^2}&\leq C\Big(q^+\|\Gamma|\cdot|\|_{L^1}+\|\Gamma\|_{L^1}\Big)\Big(\|k\|_{L^2}+\|v\|_{L^2}\Big)\|\phi\|_{H^1}
	\end{align}and
	\begin{align}\label{eq;(4.15)}
	\|\frac{\partial (\partial_\xi\w_{x})}{\partial q}\cdot\phi\|_{L^2}&\leq C\Big((q^+)^2\|\Gamma|\cdot|\|_{L^1}+q^+\|\Gamma\|_{L^1}+1\Big)\Big(\|k\|_{L^2}+\|v\|_{L^2}\Big)\|\phi\|_{H^1}.
	\end{align}
	From \eqref{eq;Gamma||_L^1}, \eqref{eq;(4.14)} and \eqref{eq;(4.15)}, it follows that
	$\frac{\partial \w_x}{\partial q}$ is a bounded linear operator from $H^1(\R)$ to $H^1(\R)$. Analogously, the other operators in \eqref{eq;operators} can be shown to be bounded linear operators from $H^1(\R)$ to $H^1(\R)$ by following the proof for $\frac{\partial \w_x}{\partial k}$ and $\frac{\partial \w_x}{\partial q}$.\par
	We have therefore shown that the right-hand side of \eqref{eq;semilinear system} is locally Lipschitz continuous near the initial data in $X$. Consequently, by the standard theory of ordinary defferential equations  in Banach spaces, the system \eqref{eq;semilinear system}–\eqref{eq;initial data of semilinear system} admits a unique local solution on some interval $[0, T]$ with $T > 0$.\\
	
	\textbf{Step 2: Extension to a global solution}\par
	To extend the local solution from Step 1 to a global solution, it suffices to show that the following quantity
	\begin{equation}\label{eq;uniformly bounded}
	\|q\|_{L^\infty}+\|\frac{1}{q}\|_{L^\infty}+\|v\|_{L^2}+\|v\|_{L^\infty}+\|k\|_{H^1}
	\end{equation}
	 remains uniformly bounded on any finite time interval. Before proving the uniform boundedness of \eqref{eq;uniformly bounded}, we first claim that
	 \begin{equation}\label{eq;claim}
	 k_\xi=\frac12 q\sin v \quad \text{and}\quad \frac{\mathrm{d}}{\mathrm{d}t}\int_{\mathbb{R}}(k^2\cos^2 \frac{v}{2}+\sin^2\frac{v}{2})q\,\mathrm{d}\xi=0.
	 \end{equation}
	 Indeed, \eqref{p(t,xi)}, \eqref{p_x(t,xi)} and \eqref{eq;semilinear system} yield
	 \begin{align}\label{eq;k_xit}
	 k_{\xi t}&=-\w_{x\xi}\notag\\
	 &=\left( e^{\lambda t}g(e^{-\lambda t}k)\cos^2\frac{v}{2}+\frac12 e^{-\lambda t}f''(e^{-\lambda t}k)\sin^2\frac{v}{2}\right)(\xi)\cdot q(\xi)\notag\\
	 &~~~-\frac{1}{2}\int_{-\infty}^{+\infty}e^{-\left|\int_{\xi}^{\xi^{\prime}}\cos^{2}\frac{v(s)}{2}q(s)\,\mathrm{d}s\right|} \cos^2\frac{v(\xi)}{2}q(\xi)\notag\\
	 &~~~\cdot\left( e^{\lambda t}g(e^{-\lambda t}k)cos^2\frac{v}{2}+\frac12 e^{-\lambda t}f''(e^{-\lambda t}k)\sin^2\frac{v}{2}\right)(\xi')\cdot q\left(\xi^{\prime}\right)\,\mathrm{d}\xi^{\prime}\notag\\
	 &=\Big( e^{\lambda t}g(e^{-\lambda t}k)\cos^2\frac{v}{2}+\frac12 e^{-\lambda t}f''(e^{-\lambda t}k)\sin^2\frac{v}{2}-\w\cos^2\frac{v}{2}\Big)q
	 \end{align}
    and
	\begin{align}\label{eq;1/2 qsinv}
	(\frac12 q\sin v)_t&=\frac12 q_t\sin v+\frac12 qv_t\cos v\notag\\
	&=\frac14 \sin^2 v\cdot q\Big((f''(e^{-\lambda t}k)e^{-\lambda t}+2e^{\lambda t}g(e^{-\lambda t}k)-2\w \Big)\notag\\
	&~~~~+(e^{\lambda t}g(e^{-\lambda t}k)-\w)\cos^2\frac{v}{2}\cdot\cos v\cdot q-\frac12 e^{-\lambda t}f''(e^{-\lambda t}k)\sin^2\frac{v}{2}\cos v\cdot q\notag\\
	&=\Big( e^{\lambda t}g(e^{-\lambda t}k)\cos^2\frac{v}{2}+\frac12 e^{-\lambda t}f''(e^{-\lambda t}k)\sin^2\frac{v}{2}-\w\cos^2\frac{v}{2}\Big)q.
	\end{align}
	From \eqref{eq;bar y xi} and \eqref{eq;initial data of semilinear system}, we observe that
	\begin{equation}\label{eq;t=0}
	k_\xi|_{t=0}=\frac{\bar{u}_x(\bar{y}(\xi))}{1+\bar{u}^2_x(\bar{y}(\xi))}=(\frac 12 q\sin v)|_{t=0}.
	\end{equation}
	Combining \eqref{eq;k_xit}, \eqref{eq;1/2 qsinv} and \eqref{eq;t=0}, we infer that 
	\begin{equation}\label{eq;k_xi=1/2qsinv}
	k_\xi=\frac12 q\sin v
	\end{equation}
	 holds for all $t\geq 0$ as long as the solution exists. Subsequently, from \eqref{eq;semilinear system} and in view of the fact that
	$$(k\w)_\xi=k_\xi\w+k\w_\xi=\frac12 \sin v\cdot q\w+k\cos^2 \frac{v}{2}\cdot q\w_x, $$ 
	we have
	\begin{align}\label{eq;4.22}
	&\frac{\mathrm{d}}{\mathrm{d}t}\int_{\mathbb{R}}(k^2\cos^2 \frac{v}{2}+\sin^2\frac{v}{2})q\,\mathrm{d}\xi \notag\\&=\int_{\R}\Big[ (2kk_t\cos^2 \frac{v}{2}-k^2\cos \frac{v}{2}\sin \frac{v}{2} \cdot v_t+\sin \frac{v}{2}\cos \frac{v}{2}\cdot v_t)\cdot q+(k^2\cos^2 \frac{v}{2}+\sin^2\frac{v}{2})\cdot q_t\Big]\,\mathrm{d}\xi\notag\\
	&=\int_{\R}\Big[-2k\w_x\cos^2 \frac{v}{2}\cdot q-\frac12 k^2\sin v\cdot q\Big(2e^{\lambda t}g(e^{-\lambda t}k)\cos^2\frac{v}{2}-2\w\cos^2\frac{v}{2}-e^{-\lambda t}f''(e^{-\lambda t}k)\sin^2\frac{v}{2}\Big)\notag\\
	&~~~~+\frac12\sin v\cdot q\Big(2e^{\lambda t}g(e^{-\lambda t}k)\cos^2\frac{v}{2}-2\w\cos^2\frac{v}{2}-e^{-\lambda t}f''(e^{-\lambda t}k)\sin^2\frac{v}{2}\Big)\notag\\
	&~~~~+(k^2\cos^2 \frac{v}{2}+\sin^2\frac{v}{2})\cdot \frac12 \sin v\cdot q\Big((f''(e^{-\lambda t}k)e^{-\lambda t}+2e^{\lambda t}g(e^{-\lambda t}k)-2\w \Big) \Big]\,\mathrm{d}\xi\notag\\
	&=\int_{\R}\Big[-2k\w_x\cos^2 \frac{v}{2}\cdot q+\frac12\sin v \cdot k^2qe^{-\lambda t}f''(e^{-\lambda t}k)+\sin v\cdot qe^{\lambda t}g(e^{-\lambda t}k)-\sin v\cdot q\w  \Big]\,\mathrm{d}\xi\notag\\
	&=\int_{\R}\Big[-2(k\w)_\xi+\frac12 k^2e^{-\lambda t}f''(e^{-\lambda t}k)\sin v \cdot q+e^{\lambda t}g(e^{-\lambda t}k)\sin v\cdot q  \Big]\,\mathrm{d}\xi\notag\\
	&=\int_{\R}\Big[-2(k\w)_\xi+k^2e^{-\lambda t}f''(e^{-\lambda t}k)k_\xi+2e^{\lambda t}g(e^{-\lambda t}k)k_\xi  \Big]\,\mathrm{d}\xi.
	\end{align}
	Defining 
	\begin{equation*}
	F(k)=\int_0^{e^{-\lambda t}k} f(s)\,\mathrm{d}s\quad\text{and}\quad G(k)=\int_0^{e^{-\lambda t}k} g(s)\,\mathrm{d}s,
	\end{equation*}
	we then obtain
	\begin{align}\label{eq;F(k)}
	k^2e^{-\lambda t}f''(e^{-\lambda t}k)k_\xi
	&=(k^2f'(e^{-\lambda t}k))_\xi-2kf'(e^{-\lambda t}k)k_\xi\notag\\
	&=(k^2f'(e^{-\lambda t}k))_\xi-2e^{\lambda t}\Big((kf(e^{-\lambda t}k))_\xi-k_\xi f(e^{-\lambda t}k)\Big)\notag\\
	&=(k^2f'(e^{-\lambda t}k))_\xi-2e^{\lambda t}(kf(e^{-\lambda t}k))_\xi+2e^{2\lambda t}(F(k))_\xi
	\end{align}
	and
	\begin{equation}\label{eq;G(k)}
	2e^{\lambda t}g(e^{-\lambda t}k)k_\xi=2e^{2\lambda t}(G(k))_\xi.
	\end{equation}
	Hence, by virtue of \eqref{eq;4.22}, \eqref{eq;F(k)} and \eqref{eq;G(k)}, we have
	\begin{align*}
	&\frac{\mathrm{d}}{\mathrm{d}t}\int_{\mathbb{R}}(k^2\cos^2 \frac{v}{2}+\sin^2\frac{v}{2})q\,\mathrm{d}\xi\notag\\ &=\int_{\mathbb{R}}\Big[ -2(k\w)_\xi+ (k^2f'(e^{-\lambda t}k))_\xi-2e^{\lambda t}(kf(e^{-\lambda t}k))_\xi+2e^{2\lambda t}(F(k))_\xi +2e^{2\lambda t}(G(k))_\xi \Big]\,\mathrm{d}\xi\notag\\
	&=0,
	\end{align*}
	which together with \eqref{eq;k_xi=1/2qsinv} establishes the claim \eqref{eq;claim}.\par
	We can now express the total energy in the new variables, which by \eqref{eq;claim} is conserved in time, namely,
	\begin{equation}\label{eq;E(t)}
	E(t)=\int_{\mathbb{R}}\left(k^2(t,\xi)\cos^2\frac{v(t,\xi)}{2}+\sin^2\frac{v(t,\xi)}{2}\right)q(t,\xi)\,\mathrm{d}\xi=E(0)\triangleq E_0
	\end{equation}
	along any solution of \eqref{eq;semilinear system}-\eqref{eq;initial data of semilinear system}. Then, \eqref{eq;claim} and \eqref{eq;E(t)} yield that
	\begin{align}\label{eq;k_L infty}
	 \sup_{\xi\in\R}|k^2(t,\xi)|&= \sup_{\xi\in\R}|\int_{-\infty}^\xi kk_\xi\,\mathrm{d}\xi-\int_\xi^{+\infty}kk_\xi \,\mathrm{d}\xi|\notag\\
	 &\leq 2\int_{\R}|kk_\xi |\,\mathrm{d}\xi\notag\\
	 &\leq 2\int_{\R}|k||\sin \frac{v}{2}\cos \frac{v}{2}|q\,\mathrm{d}\xi\notag\\
	 &\leq \int_{\R}(k^2\cos^2 \frac{v}{2}+\sin^2\frac{v}{2})q\,\mathrm{d}\xi =E_0,
	\end{align}
	which means $\|k(t)\|_{L^\infty}\leq E_0^{\frac12}$.\par
	In what follows, we estimate $\|\w\|_{L^\infty}$ and $\|\w_x\|_{L^\infty}$. However, \eqref{eq;E(t)} does not directly yield $L^\infty$ boundedness of $\|\w\|_{L^\infty}$ and $\|\w_x\|_{L^\infty}$. To overcome this, we use variable transformations and a contradiction argument. Define the flow $y(t,\xi)$ by
	\begin{equation}\label{eq;flow_def}
		\frac{\partial}{\partial t}y(t,\xi)=f'(e^{-\lambda t}k(t,y(t,\xi))),\quad y(0,\xi)=\bar{y}(\xi).
	\end{equation}
	From the definition of $\bar{y}(\xi)$, we know that $\bar{y}(\xi)\in L^\infty_{loc}$ is strictly monotonous and 
	$$ |\bar{y}(\xi_1)-\bar{y}(\xi_2)|=\left|\int_{\bar{y}(\xi_2)}^{\bar{y}(\xi_1)}1\,\mathrm{d}x\right|\leq\left|\int_{\bar{y}(\xi_2)}^{\bar{y}(\xi_1)}(1+\bar{u}_x^2)\,\mathrm{d}x\right|\leq|\xi_1-\xi_2| .$$
	It then follows that $\bar{y}$ is locally absolutely continuous. From \eqref{eq;flow_def}, there exists $T\in(0,\infty)$ such that $y(t,\xi)\in L^\infty_{loc}$ for $t\in[0,T)$, which implies $y(t,\xi)$ is locally absolutely continuous for $t\in[0,T)$. We now claim that
	\begin{equation}\label{eq;y_xi=qcos^2v/2}
	y_\xi(t,\xi)=q(t,\xi)\cos^2 \frac{v(t,\xi)}{2},\quad \text{for~all~} t\in[0,T). 
	\end{equation}
	 Indeed, from \eqref{eq;semilinear system}, \eqref{eq;k_xi=1/2qsinv} and \eqref{eq;flow_def}, we have
	 \begin{equation}\label{eq;y_{txi}}
	 y_{t\xi}=\partial_\xi f'(e^{-\lambda t}k)= e^{-\lambda t}f''(e^{-\lambda t}k)k_\xi=\frac12 e^{-\lambda t}\sin v\cdot qf''(e^{-\lambda t}k)
	 \end{equation}
	 and
	 \begin{align}\label{eq;qcos^2v/2_t}
	 (q\cos^2\frac{v}{2})_t&=q_t\cos^2\frac{v}{2}-qv_t\cos\frac{v}{2}\sin\frac{v}{2}\notag\\
	 &=\frac12 \sin v\cdot q\Big((f''(e^{-\lambda t}k)e^{-\lambda t}+2e^{\lambda t}g(e^{-\lambda t}k)-2\w \Big)\cos^2\frac{v}{2}\notag\\
	 &~~~~-\frac12 \sin v\cdot q\Big( 2e^{\lambda t}g(e^{-\lambda t}k)\cos^2\frac{v}{2}-2\w\cos^2\frac{v}{2}-e^{-\lambda t}f''(e^{-\lambda t}k)\sin^2\frac{v}{2}\Big)\notag\\
	 &=\frac12 \sin v\cdot qe^{-\lambda t}f''(e^{-\lambda t}k).
	 \end{align}
	 From \eqref{eq;bar y xi} and \eqref{eq;initial data of semilinear system},we observe that
	 \begin{equation}\label{eq;y t=0}
	 y_\xi|_{t=0}=\bar{y}_\xi(\xi)=\frac{1}{1+\bar{u}^2_x(\bar{y}(\xi))},\quad(q\cos^2\frac{v}{2})|_{t=0}=\cos^2\frac{v(0,\xi)}{2}=\frac{1}{1+\bar{u}^2_x(\bar{y}(\xi))}.
	 \end{equation}
	 Combining \eqref{eq;y_{txi}}, \eqref{eq;qcos^2v/2_t} and \eqref{eq;y t=0} yields that the claim \eqref{eq;y_xi=qcos^2v/2} holds. Then, for $t\in[0,T),~[a,b]\subset \R$, 
	 \begin{align}\label{eq;4.33}
	 \|\frac{1}{2}\int_{a}^{b}e^{-|\int_{\xi}^{\xi'}\cos^{2}\frac{v}{2}qds|}e^{\lambda t}g(e^{-\lambda t}k)\cos^{2}\frac{v}{2}q\,\mathrm{d}\xi^{\prime}\|_{L^{\infty}}
	 &\leq C\|k\|_{L^{\infty}}\|\int_{a}^{b}e^{-|y(\xi')-y(\xi)|}y_{\xi}(\xi')d\xi^{\prime}\|_{L^{\infty}}\notag\\
	 &\leq C\|k\|_{L^\infty}\int_{-\infty}^{+\infty}e^{-|s|}\,\mathrm{d}s\notag\\
	 &\leq C\|k\|_{L^\infty}.
	 \end{align}
	 According to the monotone convergence theorem, we deduce
	 that there exists a limit of $$\|\frac{1}{2}\int_{a}^{b}e^{-|\int_{\xi}^{\xi'}\cos^{2}\frac{v}{2}q\,\mathrm{d}s|}e^{\lambda t}g(e^{-\lambda t}k)\cos^{2}\frac{v}{2}q\,\mathrm{d}\xi^{\prime}\|_{L^{\infty}}$$
	  as $a\rightarrow -\infty$ and $b\rightarrow +\infty$. Consequently, from \eqref{p(t,xi)}, \eqref{eq;k_L infty} and \eqref{eq;4.33}, we obtain
	  \begin{align} \label{est;P_L infty}
	  \|\w(t)\|_{L^\infty}&=\|\frac12 \int^{+\infty}_{-\infty}e^{-|\int^{\xi'}_{\xi}\cos^2\frac{v(s)}{2}q(s)\,\mathrm{d}s|}\Big(e^{\lambda t}g(e^{-\lambda t}k)\cos^2\frac{v}{2}+\frac12 e^{-\lambda t}f''(e^{-\lambda t}k)\sin^2\frac{v}{2}\Big)q\,\mathrm{d}\xi' \|_{L^\infty}\notag\\
	  &\leq C\|\frac12 \int^{+\infty}_{-\infty}e^{-|\int^{\xi'}_{\xi}\cos^2\frac{v(s)}{2}q(s)\,\mathrm{d}s|}e^{\lambda t}g(e^{-\lambda t}k)\cos^2\frac{v}{2}q\,\mathrm{d}\xi'\|_{L^\infty}+\int^{+\infty}_{-\infty}\sin^2\frac{v}{2}q\,\mathrm{d}\xi' \notag\\
	  &\leq C(E_0^{\frac12}+E_0).
	  \end{align}
	  By the same argument, we can prove
	  \begin{equation}\label{est;P_x L infty}
	  \|\w_x(t)\|_{L^\infty}\leq C(E_0^{\frac12}+E_0).
	  \end{equation}
	  
	 In what follows, we establish the estimates of $\|q\|_{L^\infty}$ and $\|\frac1q\|_{L^\infty}$
	 over an arbitrary bounded time interval. From the system \eqref{eq;semilinear system}-\eqref{eq;initial data of semilinear system}, using \eqref{eq;k_L infty} and \eqref{est;P_L infty}, we deduce that
	 $$ |q_t|\leq C (1+E_0^{\frac12}+E_0)q,$$ 
	 from which it follows that
	 \begin{equation}\label{est;q}
	 e^{-C(1+E_0^{\frac12}+E_0)t}\leq q(t)\leq  e^{C(1+E_0^{\frac12}+E_0)t}.
	 \end{equation}
	 Meanwhile, the definition of $y(t,\xi)$ yields
	 \begin{equation}\label{est;y bounds}
	 \bar{y}(\xi)-Ct\leq y(t,\xi)\leq \bar{y}(\xi)+Ct. 
	 \end{equation}
	 Furthermore, by \eqref{eq;y_xi=qcos^2v/2} and \eqref{est;q}, we have 
	 $$\|y_\xi(t)\|_{L^\infty}\leq \|q(t)\|_{L^\infty}\leq e^{C(1+E_0^{\frac12}+E_0)t}.$$  Therefore, we get $y(t,\xi)\in H^1_{loc}$ for $t\in[0,T]$. Arguing by contradiction, we show that the time $T$ cannot be finite. Therefore, the results above hold for all time $t$ within the solution's lifespan.
	 
	 By analogy with the estimate for $\|q(t)\|_{L^\infty}$
	 , from \eqref{eq;semilinear system} and \eqref{eq;initial data of semilinear system}, we can derive the corresponding estimates for $v(t,\xi)$:
	 \begin{equation}\label{est;v}
	 \|v(t)\|_{L^\infty}\leq C(1+E_0^{\frac12}+E_0)t+\|2\arctan \bar{u}_x(\bar{y}(\xi))\|_{L^\infty}.
	 \end{equation}
	 
	 Next, we turn to establishing the estimates for $\|v(t)\|_{L^2}$.  Multiplying the second equation in \eqref{eq;semilinear system} by $2v$ and  integrating over $\xi$, we obtain that
	 \begin{align*}
	 \frac{\mathrm{d}}{\mathrm{d}t}\|v(t)\|_{L^2}^2&=2\int_{\R}v(t,\xi)\Big(2(e^{\lambda t}g(e^{-\lambda t}k)-\w)\cos^2\frac{v}{2}-e^{-\lambda t}f''(e^{-\lambda t}k)\sin^2\frac{v}{2}\Big)\,\mathrm{d}\xi \\
	 &\leq C\|v(t)\|_{L^2}\Big(\|k(t)\cos\frac{v(t)}{2}\|_{L^2}+\|\w(t)\|_{L^2}+\|v(t)\|_{L^2}\Big), \\
	 &\leq C\|v(t)\|_{L^2}\Big(E_0^{\frac12}+\|\w(t)\|_{L^2}+\|v(t)\|_{L^2}\Big),
	 \end{align*}
	 from which it follows that, for $t\leq T$,
	 $$ \|v(t)\|_{L^2}\leq C\int^t_0 \Big(E_0^{\frac12}+\|\w(\tau)\|_{L^2}+\|v(\tau)\|_{L^2}\Big) d\tau+\|v_0\|_{L^2}.$$
	 The Gronwall inequality thus implies that
	 \begin{equation}\label{est;v_L^2}
	 \|v(t)\|_{L^2}\leq e^{CT}\Big( \|v_0\|_{L^2}+CTE_0^{\frac12}+C\int^T_0\|\w(t)\|_{L^2}\,\mathrm{d}t\Big).
	 \end{equation}
	Observing the above identity, we still need to find a bound for the $L^2$ norm of $\w(t)$ to establish the estimate for 
	$\|v(t)\|_{L^2}$. Toward this goal, we proceed as in \textbf{Step~1}. Letting $C_q=e^{C(1+E_0^{\frac12}+E_0)t}$
	,from \eqref{est;q}, we have $C_q^{-1}\leq q(t)\leq C_q$.
	 We deduce that
	 \begin{align*}
	 \text{measure}\left\{\xi\in\mathbb{R}\Big| |\frac{v(\xi)}2|\geq\frac\pi4\right\} &\leq 18\int_{\left\{\xi\in\mathbb{R}|\sin^2\frac{v(\xi)}{2}\geq\frac{1}{18}\right\}}\sin^2\frac{v(\xi)}{2}\,\mathrm{d}\xi\notag\\
	 &\leq 18\int_{\R}\sin^2\frac{v(\xi)}{2}\cdot q\cdot C_q\,\mathrm{d}\xi\notag\\
	 &\leq 18E_0 C_q,
	 \end{align*}
	 which implies that, for any $\xi_1<\xi_2$, 
	 \begin{align}\label{eq;cos^2 v/2 q}
	 \int_{\xi_{1}}^{\xi_{2}}\cos^{2}\frac{v(\xi)}{2}\cdot q(\xi)\,\mathrm{d}\xi
	 &\geq\int_{\left\{\xi\in[\xi_{1},\xi_{2}]||\frac{v(\xi)}{2}|\leq\frac{\pi}{4}\right\}}\cos^2\frac{v(\xi)}{2}q(\xi)\,\mathrm{d}\xi\notag\\
	 &\geq\int_{\left\{\xi\in[\xi_{1},\xi_{2}]||\frac{v(\xi)}{2}|\leq\frac{\pi}{4}\right\}}\frac12 C_q^{-1}\,\mathrm{d}\xi\notag\\
	 &\geq \frac{\xi_{2}-\xi_{1}}{2} C_q^{-1}-9E_0.
	 \end{align}
	 Defining 
	 $$ \Gamma(\zeta)\triangleq\min\left\{1,\exp\left(9E_0-\frac{|\zeta|}{2}C_q^{-1}\right)\right\}, $$
	 we have:
	 \begin{equation} \label{eq;4.36}
	 \|\Gamma\|_{L^1}=\left(\int_{|\zeta|\leq 18E_0C_q}+\int_{|\zeta|\geq18E_0C_q}\right)\Gamma(\zeta)\,\mathrm{d}\zeta=36E_0C_q+4C_q<\infty
	 \end{equation}and
	 \begin{equation}\label{eq;4.37}
	 \|\Gamma\|_{L^2}^2=\left(\int_{|\zeta|18E_0C_q}+\int_{|\zeta|\geq18E_0C_q}\right)(\Gamma(\zeta))^2\,\mathrm{d}\zeta=36E_0C_q+2C_q<\infty.
	 \end{equation}
	 Applying Young's inequality together with \eqref{eq;cos^2 v/2 q}, \eqref{eq;4.36} and \eqref{eq;4.37}, we obtain 
	 \begin{align}\label{eq;4.40}
	 \|\w(t)\|_{L^2}&= \|\frac12 \int^{+\infty}_{-\infty}e^{-|\int^{\xi'}_{\xi}\cos^2\frac{v(s)}{2}q(s)ds|}\Big(e^{\lambda t}g(e^{-\lambda t}k)\cos^2\frac{v}{2}+\frac12 e^{-\lambda t}f''(e^{-\lambda t}k)\sin^2\frac{v}{2}\Big)q\,\mathrm{d}\xi'\|_{L^2} \\
	 &\leq C\Big( \|\Gamma\|_{L^{1}}\|k\cos\frac{v}{2}\cdot q^{\frac12}\|_{L^2}\cdot C_q^{\frac12}+\|\Gamma\|_{L^{2}}\|\sin^2\frac{v}{2}\cdot q\|_{L^{1}}\Big) \notag\\
	 &\leq C\Big( \|\Gamma\|_{L^{1}}E_0^{\frac12} C_q^{\frac12}+\|\Gamma\|_{L^{2}}E_0\Big)<\infty.
	 \end{align}
	 Plugging \eqref{eq;4.40} into \eqref{est;v_L^2}, we can see that $\|v\|_{L^2}$ remains bounded on any bounded time interval. \par
	 Combining \eqref{est;q}, \eqref{est;v} and \eqref{est;v_L^2}, and comparing with equation \eqref{eq;uniformly bounded}, we are finally reduced to controlling $\|k\|_{H^1}$ on bounded time intervals. Multiplying the first equation in \eqref{eq;semilinear system} by $2k$ and  integrating over $\xi$, we obtain that
	 \begin{align*}
	 \frac{\mathrm{d}}{\mathrm{d}t}\|k(t)\|_{L^2}^2=-2\int_{\R}k(t,\xi)\cdot \w_x(t,\xi)d\xi \leq 2\|k(t)\|_{L^2}\|\w_x(t)\|_{L^2},
	 \end{align*}
	 where, analogous to the estimate for $\|\w(t)\|_{L^2}$, 
	 \begin{align*}
	 &~\|\w_x(t)\|_{L^2}\\
	 &=\|\frac12 \Big(\int^{+\infty}_{\xi}-\int^{\xi}_{-\infty}\Big) e^{-|\int^{\xi'}_{\xi}\cos^2\frac{v(s)}{2}q(s)\,\mathrm{d}s|}\Big(e^{\lambda t}g(e^{-\lambda t}k)\cos^2\frac{v}{2}+\frac12 e^{-\lambda t}f''(e^{-\lambda t}k)\sin^2\frac{v}{2}\Big)q\,\mathrm{d}\xi' \|_{L^2}\\
	 &\leq C\Big(C_q\|\Gamma\|_{L^1}\|k\|_{L^2}+\|\Gamma\|_{L^2}\|\sin^2\frac{v}{2}\cdot q\|_{L^1}\Big)\\
	 &\leq C(C_q\|\Gamma\|_{L^1}\|k\|_{L^2}+\|\Gamma\|_{L^2}E_0).
	 \end{align*}
	 It follows that 
	 $$\frac{\mathrm{d}}{\mathrm{d}t}\|k(t)\|_{L^2}\leq C(E_0,C_q)(\|k\|_{L^2}+1), $$
	 which implies 
	 $$\|k(t)\|_{L^2}=\|\bar{u}(\bar{y}(\xi))\|_{L^2}+C\int^t_0(\|k(\tau)\|_{L^2}+1)\,\mathrm{d}\tau. $$ 
	 Applying Gronwall's inequality yields
	 \begin{align}\label{est;k_L^2}
	 \|k(t)\|_{L^2}=e^{CT}(\|\bar{u}(\bar{y}(\xi))\|_{L^2}+CT).
	 \end{align}
	 Thanks to \eqref{eq;k_xi=1/2qsinv} and \eqref{est;k_L^2}, it easily shows $\|k\|_{H^1}$ remains bounded on bounded intervals of time. This completes the proof of the theorem. 
	 
\end{proof}
We next state a lemma to be used in the subsequent section. The proof follows the argument in \cite{Constantin2007}, with the extra condition $f''\neq 0$. We omit the details.
\begin{lemm}\label{lemma4.2}
	Define
	\begin{equation*}
	\mathcal{N}\triangleq\left\{t\geq 0~|~\mathrm{measure}\{\xi\in \R~|~v(t,\xi)=-\pi\}>0\right\}.
	\end{equation*}
	Then, $\text{measure}(\mathcal{N})=0.$
\end{lemm}

\subsection{Global conservative weak solutions for the original equation} 
In this section, We will construct a global conservative weak solution to the equation \eqref{eq;u1} in the original variables 
$(t,x)$ by making use of the global solution to \eqref{eq;semilinear system}-\eqref{eq;initial data of semilinear system}.

\begin{proof}[Proof of Theorem \ref{Thm 3.3}]
	In view of Theorem \ref{Thm 3.1}, let $(k,v,q)$ be the global solution to the system \eqref{eq;semilinear system}-\eqref{eq;initial data of semilinear system} and let $y(t,\xi)$ be defined by \eqref{eq;flow_def}. Defining 
	$$ k(t,x)=k(t,\xi),\quad\text{if}~y(t,\xi)=x, $$
	we claim that $u(t,x)\triangleq e^{-\lambda t}k(t,x)$ is a  global conservative weak solution to the Cauchy problem \eqref{eq;u1} in the sense of Definition \ref{def2.1}.\par
	From \eqref{bary_def}, we have
	$$ \lim_{\xi\to\pm\infty}\bar{y}(t,\xi)=\pm\infty,$$
	which together with \eqref{est;y bounds} yields the image of the continuous map $(t,\xi)\mapsto(t,y(t,\xi)) $ is the entire half-plane $\R^{+}\times\R$. Next, \eqref{eq;y_xi=qcos^2v/2} yields the map $y\mapsto y(t,\xi)$ is non-decreasing. Moreover, if $\xi_1<\xi_2$, but $y(t,\xi_1) = y(t,\xi_2)$, then
	$$\int_{\xi_1}^{\xi_2}q(t,\zeta)\cos^2\frac{v(t,\zeta)}{2}\,\mathrm{d}\zeta=\int_{\xi_1}^{\xi_2}y_\xi(t,\zeta)\,\mathrm{d}\zeta=0 ,$$
	from which it follows that $\cos\frac{v(t,\xi)}{2}=0$ for all $\xi\in(\xi_1,\xi_2)$. Hence, by virtue of \eqref{eq;k_xi=1/2qsinv} in Theorem \ref{Thm 3.1}, we obtain
	$$u(t,\xi_2)-u(t,\xi_1)=e^{-\lambda t}(k(t,\xi_2)-k(t,\xi_1)) =e^{-\lambda t} \int_{\xi_1}^{\xi_2}\frac{1}{2}q(t,\zeta)\sin v(t,\zeta)\,\mathrm{d}\zeta=0 .$$
	This implies the map $(t,x)\mapsto u(t,x)$ is well-defined for all $t\geq 0$ and $x\in\R$. \par
	In what follows, to prove the claim above, it suffices to verify that $u(t,x)$ is a global conservative weak solution according to Definition \ref{def2.1}. Firstly, by virtue of\eqref{eq;k_xi=1/2qsinv} and \eqref{eq;y_xi=qcos^2v/2}, we have
	\begin{equation}\label{eq;k_x=sinv/1+cosv}
	k_x(t,x)=\frac{\sin v(t,\xi)}{1+\cos v(t,\xi)},\quad\text{if}~x=y(t,\xi)~\text{and}~\cos v(t,\xi)\neq-1.
	\end{equation}
	Subsequently, changing the variables and applying  \eqref{eq;E(t)} and \eqref{eq;y_xi=qcos^2v/2}, we infer that
	\begin{align}\label{eq;5.2}
	E(t)=\int_{\R}(k^{2}(t,x)+k_{x}^{2}(t,x))\,\mathrm{d}x=\int_{\{\xi|\cos v(t,\xi)>-1\}}\Big(k^2(t,\xi)\cos^2\frac{v(t,\xi)}{2}+\sin^2\frac{v(t,\xi)}{2}\Big)q(t,\xi)\,\mathrm{d}\xi\leq E_{0},
	\end{align}
	which implies that $k(t,\cdot)\in H^1(\R)$ and $\|k(t,\cdot)\|_{H^1}\leq E_0^{\frac12}$. We thus prove $u(t,\cdot)\in H^1(\R)$ at each fixed $t\geq 0$.\par
	Secondly, applying Morrey's inequality, we obtain that $k(t,x)$ is uniformly H${\rm\ddot{o}}$lder continuous in $x$ with exponent $\frac12$. 
	Using \eqref{eq;semilinear system} and \eqref{est;P_x L infty}, we infer that the map $t\mapsto k(t,y(t,\xi))$ is uniformly Lipschitz continuous along every characteristic curve $t\mapsto y(t,\xi).$ Therefore, straightforward calculations show that $k(t,x)$ is globally H${\rm\ddot{o}}$lder continuous on the entire half-plane $\R^{+}\times\R$. Next, we prove that the map $t\mapsto u(t,\cdot)$ is Lipschitz continuous from $\R^+$ into $L^2(\mathbb{R})$. Indeed, for a given point $x$ and any interval $[\tau,\tau+h]$, choose $\xi\in\R$ such that the characteristic curve passes through the point $(\tau,x)$, i.e. $y(\tau,\xi)=x.$ Using \eqref{eq;semilinear system} and \eqref{eq;k_L infty}, we have the following estimates: 
	\begin{align}
	&~\left|e^{-\lambda h}k(\tau+h, x)-k(\tau,x)\right|\notag\\
	&\leq\left|e^{-\lambda h}k(\tau+h, x)-e^{-\lambda h}k\big(\tau+h, y(\tau+h,\xi)\big)\right|+\left|e^{-\lambda h}k(\tau+h, y(\tau+h,\xi)\big)-k(\tau,x)\right|\notag\\
	&\leq e^{-\lambda h}\sup_{|y-x|\leq Ch}\left|k(\tau+h, y)-k(\tau+h,x)\right|+e^{\lambda\tau}\int_{\tau}^{\tau+h} \lambda e^{-\lambda t}\left|k(t,\xi)\right|\,\mathrm{d}t+e^{\lambda\tau}\int_{\tau}^{\tau+h} e^{-\lambda t}\left|\w_{x}(t,\xi)\right|\,\mathrm{d}t,
	\end{align}
	where $C$ is a constant depending only on $E_0$. It follows that
	\begin{align*}
	&\int_{\R}\left|u(\tau+h,x)-u(\tau,x)\right|^{2}\,\mathrm{d}x\\
	&=e^{-2\lambda \tau}\int_{\R}|e^{-\lambda h}k(\tau+h, x)-k(\tau,x)|^2\,\mathrm{d}x\\
	&\leq 3\Bigg[ e^{-2\lambda \tau}\int_{\R}e^{-2\lambda h}\sup_{|y-x|\leq Ch}\left|k(\tau+h, y)-k(\tau+h,x)\right|^2\,\mathrm{d}x +\int_\R\left(\int_\tau^{\tau+h}\lambda e^{-\lambda t}\left|k(t,\xi)\right|dt\right)^2\,\mathrm{d}x\\ &~~+\int_\R\left(\int_\tau^{\tau+h}e^{-\lambda t}\left|\w_{x}(t,\xi)\right|dt\right)^2\,\mathrm{d}x\Bigg]\\
	&\leq 3\Bigg[ e^{-2\lambda (\tau+h)}\int_{\R}\left(\int_{x-Ch}^{x+Ch}\left|k_{x}(\tau+h,y)\right|\,\mathrm{d}y\right)^{2}\,\mathrm{d}x +\int_\R\left(\int_\tau^{\tau+h}\lambda e^{-\lambda t}\left|k(t,\xi)\right|\,\mathrm{d}t\right)^2q(\tau,\xi)\cos^2\frac{v(\tau,\xi)}{2}\,\mathrm{d}\xi\\ &~~+\int_\R\left(\int_\tau^{\tau+h}e^{-\lambda t}\left|\w_x(t,\xi)\right|\,\mathrm{d}t\right)^2q(\tau,\xi)\cos^2\frac{v(\tau,\xi)}{2}\,\mathrm{d}\xi\Bigg]\\
	&\leq 3\Bigg[e^{-2\lambda (\tau+h)}\int_\R2Ch\int_{x-Ch}^{x+Ch}\left|k_x(\tau+h,y)\right|^2\,\mathrm{d}y\,\mathrm{d}x +\int_{\R}\lambda^2he^{-2\lambda \tau} \Big(\int_{\tau}^{\tau+h}\left|k(t,\xi)\right|^{2}\,\mathrm{d}t\Big)\|q(\tau)\|_{L^{\infty}}\,\mathrm{d}\xi\\
	&~~+\int_{\R}e^{-2\lambda \tau}h \Big(\int_{\tau}^{\tau+h}\left|\w_x(t,\xi)\right|^{2}\,\mathrm{d}t\Big)\|q(\tau)\|_{L^{\infty}}\,\mathrm{d}\xi\Bigg]\\
	&\leq 3\Bigg[4C^2h^{2}e^{-2\lambda (\tau+h)}\|k_{x}(\tau+h,\cdot)\|_{L^{2}}^{2} +\lambda^2he^{-2\lambda \tau}\|q(\tau)\|_{L^{\infty}}\int_{\tau}^{\tau+h}\|k(t,\cdot)\|_{L^{2}}^{2}\,\mathrm{d}t\\
	&~~+he^{-2\lambda \tau}\|q(\tau)\|_{L^{\infty}}\int_{\tau}^{\tau+h}\|\w_x(t,\cdot)\|_{L^{2}}^{2}\,\mathrm{d}t\Bigg]\\
	&\leq Ch^{2},
	\end{align*}
	which clearly implies the  Lipschitz continuity of the map $t\mapsto u(t,\cdot)$.\par 
	Thirdly, according to Lemma \ref{lemma4.2} and \eqref{eq;5.2}, we deduce that for $t\notin \mathcal{N}$,
	$$E(t)=\int_{\R}(k^{2}(t,x)+k_{x}^{2}(t,x))\,\mathrm{d}x=\int_{\{\xi|\cos v(t,\xi)>-1\}}\Big(k^2(t,\xi)\cos^2\frac{v(t,\xi)}{2}+\sin^2\frac{v(t,\xi)}{2}\Big)q(t,\xi)\,\mathrm{d}\xi=E_{0}. $$
	Since measure$(\mathcal{N})=0$, it follows from the above that $e^{2\lambda t}\int_{\mathbb{R}}(u^2(t,x)+u_x^2(t,x))\,\mathrm{d}x$ is conserved for a.e. $t\geq 0$.\par 
	Finally, owing to the fact that $L^2(\R)$ is reflexive and making use of Rademacher's theorem \cite{Aronszajn1976}, we know that the left-hand side of \eqref{eq;solution_def} is well-defined. Meanwhile, due to $k(t,\cdot)\in H^1(\R)$ and $\w_x(t,\cdot)\in L^2(\R)$, the right-hand side of \eqref{eq;solution_def} is also well-defined. We now prove that \eqref{eq;solution_def} holds for almost all $t\geq0$. To this end, we observe that for every $t>0$ and $\xi\in\R$,
	\begin{equation*}
	\frac{\mathrm{d}}{\mathrm{d}t}k\left(t,y(t,\xi)\right)=-\w_x(t,\xi),
	\end{equation*}
	where $\w_x$ is defined by \eqref{p_x(t,xi)}. For $t\notin\mathcal{N}$, the map $\xi\mapsto x(t,\xi)$ is one-to-one. Then, changing the variables and using \eqref{eq;y_xi=qcos^2v/2} and \eqref{eq;k_x=sinv/1+cosv}, we obtain that
	\begin{align*}
	\w_{x}(t,\xi)&=\frac12 \Big(\int^{+\infty}_{\xi}-\int^{\xi}_{-\infty}\Big) e^{-|\int^{\xi'}_{\xi}\cos^2\frac{v(s)}{2}q(s)\,\mathrm{d}s|}\Big(e^{\lambda t}g(e^{-\lambda t}k)\cos^2\frac{v}{2}+\frac12 e^{-\lambda t}f''(e^{-\lambda t}k)\sin^2\frac{v}{2}\Big)q\,\mathrm{d}\xi'\\
	&=\frac{1}{2}\left(\int_{\xi}^{+\infty}-\int_{-\infty}^{\xi}\right)e^{-|y(t,\xi')-y(t,\xi)|}\Big(e^{\lambda t}g(e^{-\lambda t}k)+\frac12 e^{-\lambda t}f''(e^{-\lambda t}k)k_x^2\Big)y_\xi(t,\xi') \,\mathrm{d}\xi^{\prime}\\
	&=\frac{1}{2}\left(\int_{y(t,\xi)}^{+\infty}-\int_{-\infty}^{y(t,\xi)}\right)e^{-|y(t,\xi)-x|}\Big(e^{\lambda t}g(e^{-\lambda t}k)+\frac12 e^{-\lambda t}f''(e^{-\lambda t}k)k_x^2\Big)(t,x)\,\mathrm{d}x\\
	&=\w_x(t,y(t,\xi)).
	\end{align*}
	Together with Lemma \ref{lemma4.2}, we conclude that equation \eqref{eq;solution_def} is satisfied for almost every $t\geq 0$. 
    
    Finally, the continuous dependence result follows directly from the argument in \cite{Constantin2007}. This completes the proof of this theorem.

\end{proof}

\section{Global strong solutions}\label{sec-strong}
We now turn to the global existence of strong solutions. This section addresses two distinct regimes, namely small initial data and sign-changing initial data, each utilizing a different mechanism to prevent singularity formation.
\subsection{Global strong solutions for small data}
We begin with Theorem \ref{Thm 4.1}. The proof relies on a bootstrap argument, exploiting the strong linear dissipation to bound the $H^s$-energy for small initial data.
\begin{proof}[Proof of Theorem \ref{Thm 4.1}]
Rewriting equation \eqref{eq;u1}, we obtain
\begin{align}\label{rewri}
    u_t+(f'(u)-f'(0))u_x+f'(0)u_x+\lambda u+ p_x*\left(g(u)-g(0)+\frac{f''(u)}{2}u^2_x\right)=0.
\end{align}
    Multiplying \eqref{rewri} by $u$ and taking the $H^s$-inner product, we have
    \begin{align}\label{esti}
        &\frac{1}{2}\frac{\mathrm{d}}{\mathrm{d}t}\|\Lambda^su\|^2_{L^2}+\lambda\|\Lambda^su\|^2_{L^2}\notag\\
        &\quad=-\langle\Lambda^s[(f'(u)-f'(0))u_x],\Lambda^su\rangle-\left\langle\Lambda^s\left[p_x*\left(g(u)-g(0)+\frac{f''(u)}{2}u^2_x\right)\right],\Lambda^su\right\rangle.
    \end{align}
    We now proceed to estimate each term on the right-hand side of \eqref{esti}. By the Cauchy-Schwartz inequality and Lemma \ref{lem-commu}, we get
    \begin{align}\label{force-1}
        &\left|\langle\Lambda^s[(f'(u)-f'(0))u_x],\Lambda^su\rangle\right|\notag\\
        &=\left|\langle[\Lambda^s,f'(u)-f'(0)]u_x+(f'(u)-f'(0))\Lambda^s u_x,\Lambda^su\rangle\right|\notag\\
        &\le \left\|[\Lambda^s,f'(u)-f'(0)]u_x\|_{L^2}\|\Lambda^su\right\|_{L^2}+|\langle f''(u)u_x,(\Lambda^s u)^2\rangle|\notag\\
        &\lesssim  \|u\|_{H^s}\left(\|f''(u)u_x\|_{L^{\infty}}\|\Lambda^{s-1}u_x\|_{L^2}+\|f'(u)-f'(0)\|_{H^s}\|u_x\|_{L^{\infty}}+\|f''(u)u_x\|_{L^{\infty}}\right)\notag\\
        &\lesssim \|u\|^2_{H^s}\|u_x\|_{L^{\infty}}\notag\\
        &\lesssim \|u\|^3_{H^s}.
    \end{align}
    Moreover, by Lemma \ref{lem-Fu-Fv} and Lemma \ref{lem-F0}, we obtain
    \begin{align}\label{force-2}
        &\left|\left\langle\Lambda^s\left[p_x*\left(g(u)-g(0)+\frac{f''(u)}{2}u^2_x\right)\right],\Lambda^su\right>\right|\notag\\
        &\lesssim \|u\|_{H^s}\left\|g(u)-g(0)+\frac{f''(u)}{2}u^2_x\right\|_{H^{s-1}}\notag\\
        &\lesssim \|u\|_{H^s}\left(\left\|g(u)-g(0)\right\|_{H^{s-1}}+\left\|\left(\frac{f''(u)}{2}-\frac{f''(0)}{2}+\frac{f''(0)}{2}\right)u^2_x\right\|_{H^{s-1}}\right)\notag\\
        &\lesssim \|u\|_{H^s}\left(\left\|u\right\|_{H^{s-1}}+\left\|\left(f''(u)-f''(0)\right)u^2_x\right\|_{H^{s-1}}+|f''(0)|\|u_x\|_{L^{\infty}}\|u\|_{H^s}\right)\notag\\
        &\lesssim \|u\|^2_{H^s}\left(1+\|u\|_{H^s}+\|u\|^2_{H^s}\right).
    \end{align}
    Combining \eqref{esti}, \eqref{force-1} with \eqref{force-2}, we deduce
    \begin{align}\label{esti-2}
        \frac{\mathrm{d}}{\mathrm{d}t}\|u\|^2_{H^s}+2\lambda\|u\|^2_{H^s}\le C\left(1+\|u\|_{H^s}+\|u\|^2_{H^s}\right)\|u\|^2_{H^s}.
    \end{align}
    
    Let $T^*$ be the maximal existence time of the solution $u$. To verify that $T^* = \infty$, we argue by contradiction. If $T^*$ is finite, define $\bar{T}$ by
    \begin{align}\label{barT}
        \bar{T} = \sup \left\{ T ~\middle|~ \sup_{t\in [0,T]} \| u(t)\|_{H^{s}}^{2} + \lambda\int_{0}^{T} \| u(t)\|_{H^{s}}^{2}\,\mathrm{d}t \le 2\varepsilon \right\}.
    \end{align}
    Obviously, $\bar{T}\le T^*$. Assuming $\bar{T}< T^*$, in view of \eqref{esti-2} and \eqref{barT}, for $t<\bar{T}$, it follows that
    \begin{align*}
        \frac{\mathrm{d}}{\mathrm{d}t}\|u\|_{H^{s}}^{2} + 2\lambda\|u\|_{H^{s}}^{2}\leq C\left(1+(2\varepsilon)^{\frac{1}{2}}+(2\varepsilon)\right)\| u\|_{H^{s}}^{2}.
    \end{align*}
    Let $C\left(1+(2\varepsilon)^{\frac{1}{2}}+(2\varepsilon)\right)<\lambda$, provided that $\varepsilon$ is sufficiently small. Consequently,
    \begin{align}
        \frac{\mathrm{d}}{\mathrm{d}t}\|u\|_{H^{s}}^{2} + \lambda\|u\|_{H^{s}}^{2}\leq 0.
    \end{align}
    Upon integrating over $[0, \bar{T}]$ with respect to $t$ and applying the initial condition $\|\bar{u}\|_{H^s}^2 \le \frac{\varepsilon}{2}$ when $\varepsilon$ is sufficiently small, we deduce that
    \begin{align*}
        \sup_{t\in [0,\bar{T}]}\| u(t)\|_{H^{s}}^{2} + \lambda\int_{0}^{\bar{T}}\| u(t)\|_{H^{s}}^{2}\,\mathrm{d}t\leq 2\| \bar{u}\|_{H^{s}}^{2}\leq \varepsilon < 2\varepsilon,
    \end{align*}
    which contradicts \eqref{barT}. 

    Therefore, we have $\bar{T}=T^*$ and
        \begin{align}\label{uT*}
        \sup_{t\in [0,T^*)}\| u(t)\|_{H^{s}}^{2} + \lambda\int_{0}^{T^*}\| u(t)\|_{H^{s}}^{2}\,\mathrm{d}t\leq 2\varepsilon.
    \end{align}
    Now, denote $T_{\delta}:=T^*-\delta\in [0,T^*)$, where $\delta>0$ is sufficiently small such that $\delta<\frac{C}{\sqrt{2\varepsilon}}$. In light of \eqref{uT*}, it gives
    \begin{align*}
        \|u(T_{\delta})\|^2_{H^s}\le 2\varepsilon.
    \end{align*}
    Taking $u(T_\delta)$ as the initial data and invoking the local well-posedness theory established in \cite{Novruzov2019,Tian2014}, we deduce that there exists a unique solution $u$ with lifespan $T_{\text{life}}$ satisfying
    \begin{align*}
    T_{\text{life}}\ge\frac{C}{\|u(T_{\delta})\|_{H^s}}\ge \frac{C}{\sqrt{2\varepsilon}},
    \end{align*}
    which implies that
    \begin{align*}
        T_{\delta}+T_{\text{life}}\geq T_{\delta}+\frac{C}{\|u(T_{\delta})\|^2_{H^s}}>T_{\delta}+\delta=T^*.
    \end{align*}
    This contradicts the definition of $T^*$. Therefore, we conclude that $T^*=\infty$, which yields the desired result.
\end{proof}

\subsection{Global strong solutions for sign-changing data}
In this subsection, we prove Theorem \ref{Thm 4.6}. For sign-changing data, combining the sign-preservation of momentum $m$ along characteristics with the weighted energy conservation yields a uniform lower bound on $u_x$ to prevent blow-up.

We rewrite equation \eqref{eq;u1} in the following equivalent form:
\begin{equation}\label{eq;m}
	\left\{\begin{array}{l}
	m_{t}+f'(u)m_x+\lambda m-\frac 12 f'''(u)u_x^3+2f''(u)u_xm-2f''(u)uu_x+g'(u)u_x=0,\\
        m=u-u_{xx},\\
	u|_{t=0}=\bar{u}.
	\end{array}\right.
	\end{equation}
Recall the differential equation considered in \eqref{eq;flow_def}:
\begin{equation}\label{eq;y_flow}
		\frac{\partial}{\partial t}y(t,x)=f'(u(t,y(t,x))),\quad y(0,x)=\bar{y}(x).
	\end{equation}
    Using \eqref{eq;bar y xi}, we obtain the following lemma.
\begin{lemm}\label{le4.4}
	Assume $f, g \in C^\infty (\R)$. Let $\bar{u}\in H^s(\mathbb{R})$ with $s>\frac 32$. Then \eqref{eq;y_flow} has a unique solution $ y \in C^1([0,T)\times \mathbb{R};\mathbb{R})$. Moreover, the map $y(t,\cdot)$ is an increasing diffeomorphism of $\mathbb{R}$ with  
		$$y_x(t,x)=\frac{1}{1+\bar{u}_x^2(\bar{y}(x))}\exp\Big(\int_{0}^{t} (f'(u))_x(s,y(s,x))\,\mathrm{d}s\Big)>0,~\forall(t,x)\in [0,T)\times \mathbb{R}.$$ 
	\end{lemm}

\begin{lemm}\label{le4.5}
    Assume $f, g \in C^\infty (\R)$ satisfying 
    \begin{equation}\label{eq;con}
        f'''\equiv 0,\quad g'(u)=2f''(u)u.
    \end{equation}
    Let $\bar{u}\in H^s(\R)$ with $s>\frac 32$, and let $T > 0$ be the maximal existence time of the corresponding solution $u$ to \eqref{eq;m}. Then we have
    \begin{equation}
        m(t,y(t,x))y_x^2(t,x)=\bar{m}(\bar{y}(x))\bar{y}_x^2(x)e^{-\lambda t},\quad\forall(t,x)\in [0,T)\times \mathbb{R},
    \end{equation}
    where $\bar{m}=\bar{u}-\bar{u}_{xx}.$
\end{lemm}
\begin{proof}
    Using \eqref{eq;m} and condition \eqref{eq;con}, we have
    \begin{align*}
        \frac{\mathrm{d}}{\mathrm{d}t}(m(t,y(t,x))y_x^2)&=m_t(t,y)y_x^2+m_x(t,y)f'(u(t,y))y_x^2+2m(t,y)y_xy_{tx} \\
        &=-\lambda m(t,y)y_x^2+\Big(\frac{1}{2}f'''(u)u_x^2+2f''(u)uu_x-g'(u)u_x \Big)y_x^2\\
        &=-\lambda m(t,y)y_x^2.
    \end{align*}
    Therefore, by solving the above ODE, we complete the proof.
\end{proof}
    Combining Lemmas \eqref{le4.2}, \eqref{le4.3}, \eqref{le4.4}, \eqref{le4.5} and following the method in \cite{Yin2009}, we give a brief proof of Theorem \ref{Thm 4.6}.

\begin{proof}[Proof of Theorem \ref{Thm 4.6}]
    Let $T > 0$ be the maximal existence time of the solution $u(t, x)$ to equation \eqref{eq;u1}, as given by Lemma \ref{le4.2}.
    Applying Lemmas \ref{le4.4} and \ref{le4.5}, for $t\in[0,T)$, we have
    \begin{equation*}
        m(t,x)\leq 0,\quad \text{if}~x\leq y(t,x_0)\quad \text{and}\quad  m(t,x)\geq 0,\quad \text{if}~x\geq y(t,x_0).
    \end{equation*}
    By the Sobolev imbedding theorem and the fact that $e^{\lambda t}\|u(t)\|_{H^1}=\|\bar{u}\|_{H^1}$, we deduce that
    \begin{equation*}
        u_x(t,x)\geq-\|u(t,\cdot)\|_{L^\infty}\geq-\frac{1}{\sqrt{2}}\|\bar{u}\|_{H^1},\quad \forall(t,x)\in[0,T)\times\R.
    \end{equation*}
    Combining with Lemma \ref{le4.3}, we deduce that $T=\infty$. The proof of the decay of global solutions follows similarly to that in \cite{Yin2009}, and is therefore omitted.
    
\end{proof}

	\smallskip
	\noindent\textbf{Acknowledgments.}~~This work was
	supported by the National Natural Science Foundation of China (No.12571261).

	\noindent\textbf{Data Availability.}
	No data were used for the research described in the article.

	\phantomsection
	\addcontentsline{toc}{section}{\refname}
	\bibliographystyle{abbrv} 
	\bibliography{refref1}

\end{document}